\documentclass{amsart}

\usepackage{hyperref}

\usepackage{esint}
\usepackage{amssymb}
\usepackage{tikz}
\usepackage{graphicx}
\usepackage{amsrefs,enumitem,pgfkeys}

\usepackage{fancyhdr}
\usepackage{time}
 
\newtheorem{theorem}{Theorem}
\newtheorem{lemma}[theorem]{Lemma}
\newtheorem{corollary}[theorem]{Corollary}
\newtheorem{proposition}[theorem]{Proposition}
\theoremstyle{definition}

\newtheorem{remark}[theorem]{Remark}

\newtheorem{definition}[theorem]{Definition}

\newcommand{\eref}[1]{(\ref{e.#1})}
\newcommand{\tref}[1]{Theorem \ref{t.#1}}
\newcommand{\lref}[1]{Lemma \ref{l.#1}}
\newcommand{\pref}[1]{Proposition \ref{p.#1}}
\newcommand{\cref}[1]{Corollary \ref{c.#1}}

\newcommand{\sref}[1]{Section \ref{s.#1}}
\newcommand{\partref}[1]{\ref{part.#1}}
\newcommand{\dref}[1]{Definition \ref{d.#1}}
\newcommand{\rref}[1]{Remark \ref{r.#1}}

\newcommand{\aref}[1]{Assumption \ref{a.#1}}

\numberwithin{theorem}{section}
\numberwithin{equation}{section}

\newcommand{\Z}{\mathbb{Z}}
\newcommand{\R}{\mathbb{R}}

\newcommand{\T}{\mathbb{T}}
\newcommand{\grad}{\nabla}

\def\XXint#1#2#3{{\setbox0=\hbox{$#1{#2#3}{\int}$ }
\vcenter{\hbox{$#2#3$ }}\kern-.6\wd0}}

\newcommand{\ep}{\varepsilon}
\newcommand{\osc}{\mathop{\textup{osc}}}

\begin{document}

\title[Quantitative homogenization of shape optimizers]{Quantitative homogenization of principal Dirichlet eigenvalue shape optimizers}
\author{William M Feldman}
\email{feldman@math.utah.edu}
\address{Department of Mathematics, University of Utah, Salt Lake City, USA}

\begin{abstract}
We apply new results on free boundary regularity from our paper \cite{FeldmanReg} to obtain a quantitative convergence rate for the shape optimizers of the first Dirichlet eigenvalue in periodic homogenization.  We obtain a linear (with logarithmic factors) convergence rate for the optimizing eigenvalue. Large scale Lipschitz free boundary regularity of almost minimizers from \cite{FeldmanReg} is used to apply the optimal $L^2$ homogenization theory in Lipschitz domains of Kenig, Lin and Shen \cite{KenigLinShen}. A key idea, to deal with the hard constraint on the volume, is a combination of a large scale almost dilation invariance with a selection principle argument.
\end{abstract}
\maketitle

\section{Introduction} 
Consider the principal Dirichlet eigenvalue of an elliptic operator $-\grad \cdot (a(x) \grad \cdot)$
 \begin{equation}\label{e.eigenvaluedef}
  \lambda_1(V,a) = \inf \{ \int_V a(x) \grad v \cdot \grad v \ dx : \ v \in H^1_0(V), \ \|v\|_{L^2(V)} \geq 1\}.
  \end{equation}
 When $a(x) = \textup{id}$ is constant coefficient the classical Faber-Krahn inequality says that the domain of a fixed volume which minimizes the principal Dirichlet eigenvalue is a ball: for any $V$ open bounded in $\R^d$
 \[\lambda_1(V,\textup{id}) \geq \lambda_1(B,\textup{id})  \ \hbox{ when $B$ is a ball with $|B| = |V|$.}\]
 If $a(x) = \bar{a}$ a constant elliptic matrix then the minimizers are $\bar{a}$-ellipsoids $E = \bar{a}^{\frac{1}{2}}B$ instead.
 
 Now consider the problem of minimizing the Dirichlet eigenvalue over the domain $V$ when $a(x)$ is $\Z^d$-periodic and uniformly elliptic.   More precisely, consider a domain $U$ so that
 \begin{equation}\label{e.shapeopt}
   \lambda_1(U,a) \leq \lambda_1(V,a) \ \hbox{ for all $V$ quasi-open with } \ |V| \leq m.
  \end{equation}
    The class of quasi-open sets is simply a natural weak class to define this minimization problem on, it will be explain more precisely later.  
    
    It is the classical content of elliptic homogenization theory that the operator $- \grad \cdot (a(x)\grad )$ is well approximated at large scales by a homogeneous operator $- \grad \cdot (\bar{a} \grad )$.  Fitting this theme, the goal of this paper is to exploit the periodic structure of the matrix field $a(x)$ to find a large scale asymptotic expansion of the principal eigenvalue.  One might conjecture an expansion of the form
    \begin{equation}\label{e.formal-expansion}
     \lambda_1(U,a)m^{\frac{2}{d}} = \lambda_1(E,\bar{a})m^{\frac{2}{d}} + A_1m^{-\frac{1}{d}} + A_2m^{-\frac{2}{d}} +\cdots \ \hbox{ valid for } \ m \gg 1
     \end{equation}
    where $E = \bar{a}^{\frac{1}{2}}B$ is an $\bar{a}$-ellipsoid of the same volume as $U$.  We are not close to establishing a full expansion like this, instead what we show in this paper is an error estimate which is (almost) linear
    \[ m^{\frac{2}{d}}|\lambda_1(U,a) - \lambda_1(E,\bar{a})| \leq Cm^{-\frac{1}{d}} |\log m|^{p} \ \hbox{ valid for } \ m \gg 1\]
    where $p>d+4$ (not optimal) and $C$ depends on $p$ and other universal constants of the problem made precise later.

 It is well appreciated in the field that the hard constraint on the volume in the shape optimization problem \eref{shapeopt} is a major difficulty.  We will rely, following a standard idea in the literature, on the relationship between \eref{shapeopt} and a related \emph{augmented problem} which is just a Lagrange multiplier / soft constraint version of the problem and is closely related to Bernoulli-type free boundary problems.  Call the augmented shape functional $J_\mu$ to be
   \begin{equation}
    J_\mu(U,a) = \lambda_1(U,a) + \mu |U|.
    \end{equation}
   We are still interested in the large volume case, which corresponds here to small penalization $\mu>0$.

  \begin{theorem}\label{t.main-aug}
Suppose $U$ is a minimizer of $J_\mu$ over quasi-open sets so that, in particular, $|U| \sim_{\Lambda,d} \mu^{-\frac{d}{d+2}}$.  Then $U$ has the following regularity properties depending on the universal parameters $(\Lambda,d,\|\grad a\|_\infty)$:
\begin{enumerate}[label = (\roman*)]
\item The principal eigenfunction $u_U$ is Lipschitz with (scaling relative) universal constant.
\item The principal eigenfunction $u_U$ is non-degenerate at its free boundary with (scaling relative) universal constant.
\item There is a Lipschitz domain $U_- \subset U$ with (scaling relative) universal parameters such that $d_H(\partial U_-,\partial U) \leq C$ universal.
\end{enumerate}
Finally, for any $\eta>0$, there is a constant $C \geq 1$ depending on universal parameters and $\eta$ so that
\[ |U|^{\frac{2}{d}}|\lambda_1(U,a) - \lambda_1(E,\bar{a})| \leq C |U|^{-\frac{1}{d}}|\log(2+|U|)|^{\frac{1}{2}+\eta}\]
and
\[ \inf_{|E| = |U|} \frac{|U \Delta E|}{|U|} \leq C|U|^{-\frac{1}{2d}}|\log (2+|U|)|^{\frac{1}{4}+\eta}\]
where the above infimum is over $\bar{a}$-ellipsoids of the form $E = \bar{a}^{\frac{1}{2}}B$ for a ball $B$.
 \end{theorem}
 The precise versions of the regularity statements in \tref{main-aug} with all the scalings can be found below in \sref{scalings-and-regularity} and in the statement of \tref{augmented-full}.

For volume constrained minimizers we are not able to establish such strong regularity properties.  Nonetheless, we can still show almost the same conclusion about the convergence of the eigenvalue $\lambda_1(U,a)$ and the convergence in measure to an ellipsoid. The issue is, although all augmented minimizers are volume constrained minimizers, the reverse may not be true.  However the $1$-to-$1$ correspondence does hold for dilation invariant operators, and we are able to take advantage of the approximate dilation invariance of $J_\mu$ minimizers implied by \tref{main-aug}.  Our analysis centers on the multi-valued mapping from $\mu$ to the values of the volume taken by $J_\mu$ minimizers. This mapping is monotone and may have discontinuities where certain volumes are missed.  By a convexity/dilation argument at these singular values, using the dilations, we can show that volume constrained minimizers have small $J_\mu$ energy deficit for a good choice of $\mu$.  Then we apply a penalization/selection principle argument to find a $J_{g(x)}$ minimizer $\Omega$ which is close to $U$ for a non-constant penalization function $g(x) \approx \mu$.  This penalized minimizer $\Omega$ does have all the nice regularity properties of \tref{main-aug} and we can take advantage of that.

For the purposes of stating our main result on volume constrained minimization let us define the scaled energy deficit
\[ \delta_1(U,a) = |U|^{\frac{2}{d}}(\lambda_1(U,a) - \min_{|V| = |U|}\lambda_1(V,a)).\]

\begin{theorem}\label{t.main}
For any $p>d+4$ there is a constant $C \geq 1$ depending on $(\Lambda,d,\|\grad a\|_\infty,p)$ so that
\[ m^{\frac{2}{d}}|\inf_{|V| = m}\lambda_1(V,a) - \min_{|V| = m}\lambda_1(V,\bar{a})| \leq C m^{-\frac{1}{d}}|\log (2+m)|^{p}\]
and for any domain $U$ with $\delta_1(U,a) \leq \frac{1}{2}$
\[ \inf_{|E| = |U|} \frac{|U \Delta E|}{m} \leq C\left[\delta_1(U,a)^{\frac{1}{2}}|\log \delta_1(U,a)|^{\frac{p}{2}}+m^{-\frac{1}{2d}}|\log (2+m)|^{\frac{p}{2}}\right]\]
where the above infimum is over $\bar{a}$-ellipsoids of the form $E = \bar{a}^{\frac{1}{2}}B$ for a ball $B$.
\end{theorem}
\begin{remark}We do not know if these rates are optimal.  Linear convergence of $r^2|\lambda_1(rU,a) - \lambda_1(rU,\bar{a})|$ to zero is certainly the best possible for a fixed domain $U$ as $r \to \infty$.  The difference of the shape optimizing eigenvalues could be higher order.  We do not currently have a prediction for the precise form of the next order term in the formal expansion \eref{formal-expansion}.  The square root convergence of the domains in measure is probably optimal for sets with positive energy deficit $\delta_1(U,a)>0$, but the actual minimizer, if it exists, could behave better.
\end{remark}
\begin{remark}
We do prove a certain regularity estimate for volume constrained minimizers $U_*$, namely there is a domain $\Omega_*$ with all the regularity properties specified in \tref{main-aug} such that the measure difference $m^{-1}|U_* \Delta \Omega_*|$ is linearly small $m^{-\frac{1}{d}}$ in the length scale up to logarithmic factors.  See \pref{hard-constraint-replacement} below for a precise statement.
\end{remark}
\begin{remark}
Note that we do not technically discuss minimizers, only sets with small energy deficit.  Therefore we do not need to, nor do we, prove existence of volume constrained minimizers on $\R^d$.  
\end{remark}

\begin{remark}
The main results \tref{main-aug} and \tref{main} can be easily translated in the more common $\frac{x}{\ep}$ scaling in homogenization theory where, for example, the eigenvalue estimate in \tref{main} would translate to
\[ |\inf_{|V| = 1}\lambda_1(V,a(\tfrac{\cdot}{\ep})) - \inf_{|V| = 1}\lambda_1(V,\bar{a})| \leq C\ep |\log (2 + \ep^{-1})|^{p}. \]
Notice though that if we want to consider all volumes $|V| = m \gg \ep^d$ instead of just $|V| = 1$ we will still need to grapple with the scalings in the volume anyway, this is why we just take $\ep = 1$ and consider $m \gg 1$ large.
\end{remark}

\subsection{Summary of the paper}
In \sref{known-results} and \sref{L2theory} we build up some practical technical results and summarize important results from the literature that will come into play in the proof.  In order to get directly to the proof of the main results one should skip ahead to \sref{augmented-reg} where the proof of our main result about augmented functional minimizers \tref{main-aug} is given, a detailed outline of the arguments can also be found in \sref{augmented-reg}.  Finally in \sref{relation} we explain how to use results on the augmented functional to study the hard constraint minimization problem.

\subsection{Literature}
The idea, originated by Avellaneda and Lin \cite{AvellanedaLin}, of using large scale regularity theory inherited from the effective problem to prove optimal rates of convergence is, by now, a standard and important idea in homogenization theory \cite{AvellanedaLin,ArmstrongSmart,GloriaNeukammOtto,AKM}.  To our knowledge this is one of the first examples where such an idea has been applied to an oscillatory interface / free boundary / shape optimization problem, the only other example we are aware of is a very nice recent work by Aleksanyan and Kuusi \cite{AleksanyanKuusi} on the obstacle problem in random media.  The deep challenge of many interface problems is that there is a high degree of non-uniqueness, there are major differences between, minimal supersolutions, maximal subsolutions, local energy minimizers, and global energy minimizers (almost minimizers are closest to this class).  For example, as seen in \cite{CaffarelliLee,KimPinning,FeldmanSmart,FeldmanPinning}, the homogenized problem associated with local minimizers of Bernoulli-type functionals can be of a completely different type compared to the effective problem for minimizers or almost minimizers.  Thus this eigenvalue shape optimization problem is a rather special scenario where we can take advantage of the variational / almost minimizer property of the solutions in question.

There is a huge literature on shape optimization problems for functionals involving domain eigenvalues (Dirichlet, Neumann, Steklov etc.).  In that context we are considering one of the most central, well-studied, and classical problems, the optimization of the principal Dirichlet eigenvalue. We cannot be exhaustive but we will try to review relevant results.  Buttazzo and Dal Maso \cite{ButtazzoDalMaso} gave the original proof of existence of minimizers in the quasi-open class for a large class of eigenvalue shape optimization problems. 

The first result we are aware of on the higher regularity of optimal sets was by Brian\c{c}on and Lamboley \cite{BrianconLamboley} who showed an almost minimality property for shape optimizers of the first Dirichlet eigenvalue with a domain constraint and then used the Alt and Caffarelli \cite{AltCaffarelli} regularity strategy to obtain domain regularity.  This has inspired a lot of work since.  

More recently there has been work on variable coefficient problems and more complicated functionals of the eigenvalues including Lipschitz estimates of the eigenfunction and higher domain regularity: Russ, Trey, and Velichkov \cite{RussTreyVelichkov} consider a special form of Laplacian with drift $-\Delta + \grad \Phi \cdot \grad$, Kriventsov and Lin \cite{KriventsovLin1,KriventsovLin2} prove regularity for minimizers of general functionals of the higher Laplacian eigenvalues, Lamboley and Sicbaldi \cite{LamboleySicbaldi} considered shape optimizers for Laplace-Beltrami principal eigenvalues on Riemannian manifolds with boundary, Trey \cite{Trey,Trey2} considered variable coefficient problems for the sum of the first $k$ Dirichlet eigenvalues.  Not directly in the context of eigenvalue problems David, Toro, Smit Vega Garcia and Engelstein  \cite{DavidEngelsteinSVGToro} have also proven a Lipschitz estimate for almost minimizers of two-phase Bernoulli-type functionals.  

Finally we mention the work of Brasco, De Philippis and Velichkov \cite{BDPV} on the sharp quantitative Faber-Krahn inequality.   In a few ways our paper bears similarities and is inspired by their work.  Vaguely speaking the line of their argument, similar to our paper, involves constructing some almost minimizers of a Bernoulli-type functional which are close to the ball Faber-Krahn minimizers and inherit regularity. 

We must note that all of these regularity results, Lipschitz continuity of the eigenfunction, and domain regularity, are all \emph{small scale} regularity results.  If they are applied in our context with a large volume constraint small $\mu$ the regularity \emph{will not scale correctly} in the volume.  In order to obtain large scale regularity we must take advantage of the homogenization structure of the problem.

\subsection{Acknowledgments} The author was supported by the NSF grant DMS-2009286.  The author would like to thank Farhan Abedin for several in depth conversations which greatly helped in shaping the paper.

\section{Some known results and other preliminaries}\label{s.known-results}

\subsection{Notations, conventions and universal hypotheses}
We make precise the assumptions on the coefficients which will be in place throughout the paper, unless otherwise specified in some few locations.

The hypotheses on the coefficient field $a: \R^d \to M_{d \times d}^{sym}(\R)$:
\begin{enumerate}[label = (a\arabic*)]
\item \label{a.a1} (Ellipticity bounds)
\begin{equation*}\label{e.aellipticity}
\Lambda^{-1} I\leq a(x) \leq \Lambda I.
\end{equation*}
\item \label{a.a2} (Periodicity) $a(x)$ is $\Z^d$-periodic.
\item \label{a.a3} (Lipschitz) $\|\grad a\|_\infty < + \infty$.
\end{enumerate}
\aref{a3} is used in the proof of non-degeneracy in \cite{FeldmanReg}, possibly it could be weakened, see comments there in Section 1.1.
\begin{enumerate}[label=$\bullet$]
\item  Constants $C$ and $c$ in the text below which depend at most on $d$, $\Lambda$, $\|\grad a\|_\infty$ are called \emph{universal}.  Such constants may change meaning from line to line without mention.  If we intend to fix the value of such a universal constant for a segment of the argument we may denote it as $c_0$, $C_0$, $c_1$ etc.
\end{enumerate}
Notations:
\begin{enumerate}[label = $\bullet$]
\item Balls are denoted by $B$, $B_r$ (if the radius/center does not need denotation) or $B_r(x)$. 
\item The ellipsoid associated with a constant elliptic matrix $\bar{a}$ is defined
\begin{equation}
 E_1(0;\bar{a}) = \textup{det}(\bar{a})^{-1/2}\bar{a}^{\frac{1}{2}}B_1(0).
 \end{equation}
Translations and dilations are denoted $E_r(x;\bar{a}) = x+rE_1(0;\bar{a})$.  Generic ellipsoids in this family may be denoted with $E$ and we call these $\bar{a}$-ellipsoids.
\item Lebesgue measure of a measurable set $U$ is denoted $|U|$.  
\item Averaged $L^p$ norms \[\|f\|_{{\underline{L}}^p(U)} = (\frac{1}{|U|} \int_U |f|^p \ dx)^{1/p}.\]
\item We denote the Hausdorff distance between compact subsets of $\R^d$ by $d_H(E,F)$.
\end{enumerate}

\begin{definition}
Say that a modulus of continuity $\omega: [0,\infty) \to [0,\infty)$ is a Dini modulus if
\[ \sum_{k=1}^\infty \omega (\theta^k) < +\infty\]
for any $\theta \in (0,1)$.
\end{definition}

\subsection{Quasi-open sets and existence of shape optimizers in compact spaces}
The class of quasi-open sets is a natural weak space to consider eigenvalue shape optimization problems.  The reason is, essentially, that quasi-open sets are the positivity sets of functions in $H^1(\R^n)$.  We give a brief summary of relevant information, but for a comprehensive presentation see \cite{VelichkovBook}.

The capacity of a set $\Omega$ in $\R^n$ is defined
\[ \textup{cap}(\Omega) = \inf\{ \|u\|_{H^1}: \ u \in H^1(\R^d) \ \hbox{ and $u \geq 1$ in a neighborhood of $\Omega$}\}.\]
A property is said to hold quasi-everywhere or q.e. if it holds on the complement of a set of zero capacity.

A set $U$ is called quasi-open if there is a decreasing collection of open sets $\Omega_n$ such that $U \cup \Omega_n$ is open and $\lim_{n \to \infty} \textup{cap}(\Omega_n) = 0$.

Sobolev functions $u \in H^1(\R^d)$ can be modified on a set of measure zero to be quasi-continuous, i.e. continuous outside of a capacity zero set.  And then one can see that $\{u>0\}$ is a quasi-open set, and vice versa every quasi-open set is the positivity set of a Sobolev function.

The Sobolev space $H^1_0(U)$ can be made sense of for quasi-open sets as
\[ H^1_0(U) = \{u \in H^1(\R^d): \ u = 0 \hbox{ q.e. on } \ \R^d \setminus U\}.\]

Mostly we do not need to concern ourselves with these notions, we can take for granted the existence of open $J_\mu$ minimizers on a large torus $N\T^d$ by applying, for example, the existence and regularity theorem of Lamboley and Sicbaldi \cite{LamboleySicbaldi}.

\subsection{Scaling normalized regularity bounds}\label{s.scalings-and-regularity} In this section we define a variety of regularity quantities that respect the scaling invariances of the eigenvalue problem.  In later sections we will prove that minimizers have universal bounds on these scale invariant quantities and/or prove various estimates assuming bounds on these quantities.  The definitions are listed more or less in the order that they will be proved.
\begin{definition}
Given a quasi-open domain $\Omega \subset\R^d$ call $u_{\Omega}$ to be the first Dirichlet eigenfunction satisfying
\[ u_{\Omega}(\cdot,a) = \textup{argmin} \{ \int_\Omega \frac{1}{2} \grad v \cdot a(x) \grad v \ dx \ : v \in H^1_0(\Omega), \ \|v\|_{L^2(\Omega)} \geq 1 \}. \]
Call $\tilde{\Omega} = |\Omega|^{-1/d}\Omega$ the dilation of $\Omega$ with unit volume and
\[ \tilde{u}_{\Omega}(\cdot,a) = \textup{argmin} \{ \int_{\tilde{\Omega}} \frac{1}{2} \grad v \cdot a(|\Omega|^{1/d}x) \grad v \ dx \ : v \in H^1_0(\Omega), \ \|v\|_{L^2(\Omega)} \geq 1 \}\]
so that $u_{\Omega}(x,a) = |\Omega|^{-1/2}\tilde{u}_\Omega(|\Omega|^{-1/d}x,a)$.
\end{definition}

\begin{definition}\label{d.L-Lipschitz}
Say that a domain $\Omega$ has the scale invariant principal eigenfunction $L$-Lipschitz estimate property for the operator $-\grad \cdot (a(x)\grad\cdot)$ if the principal eigenfunction $u_{\Omega}(\cdot,a)$ has
\[ |\Omega|^{\frac{1}{2}+\frac{1}{d}}\sup_{\Omega}|\grad u_{\Omega}| = \sup_{\tilde{\Omega}}|\grad \tilde{u}_{\Omega}| \leq L\]
\end{definition}
\begin{definition}\label{d.ell-non-degen}
Say that a domain $\Omega$ has the scale invariant $\ell$-non-degeneracy property for the operator $-\grad \cdot (a(x)\grad\cdot)$ if the principal eigenfunction $u_{\Omega}(\cdot,a)$ has weak non-degeneracy
\[ \ u_{\Omega}(x) \geq |\Omega|^{-\frac{1}{2}-\frac{1}{d}}\ell r\ \hbox{ for all } \ x \in  \Omega \ \hbox{ and } \ 0 < r \leq d(x,\partial \Omega)\]
and strong non-degeneracy
\[ \sup_{y \in B_r(x)} u_{\Omega}(y) \geq |\Omega|^{-\frac{1}{2}-\frac{1}{d}}\ell r\ \hbox{ for all } \ x \in \partial \Omega \ \hbox{ and } \ r \leq |\Omega|^{1/d}.\]
\end{definition}

\begin{definition}\label{d.scale-inv-densities} 
Say that a domain $\Omega$ has inner/outer density bound $\kappa_0$ if, for all $x_0 \in \partial \Omega$ and $0 < r \leq |\Omega|^{\frac{1}{d}}$
\[ \kappa_0 \leq \frac{|U \cap B_r(x_0)|}{|B_r(x_0)|} \leq 1- \kappa_0\]
and we make note that the scale invariant global density quantity is also bounded
\begin{equation}\label{e.kappadef}
 \kappa_\Omega = \sup_{z \in \Omega, r >0} \frac{|B_r(z) \cap \Omega|}{|B_{r/2}(z) \cap \Omega|} + \sup_{z \in \partial \Omega, r>0} \frac{|B_r|}{|B_r(z) \setminus \Omega|} \leq C(d)\kappa_0^{-1}
 \end{equation}
 a proof of this technical relationship can be found in \sref{kappa-relation}.
\end{definition}

\begin{definition}\label{d.perimeter-hyp}
Say that a domain $\Omega$ has scale invariant boundary strip area bound $P$ if 
\[
  |\{x \in \Omega : d(x,\partial \Omega) \leq t\}| \leq P|\Omega|^{\frac{d-1}{d}}t \ \hbox{ for all } \ t>0.
\]
 
\end{definition}
Note that \dref{perimeter-hyp} implies a perimeter bound by $P|\Omega|^{\frac{d-1}{d}}$ and together with the density estimates \dref{scale-inv-densities} it implies a Hausdorff $\mathcal{H}^{d-1}$ estimate of $\partial \Omega$.

\begin{definition}\label{d.weakly-regular}
Say that a domain $\Omega$ is weakly regular if it satisfies \dref{scale-inv-densities} and \dref{perimeter-hyp}.
\end{definition}

\begin{definition}\label{d.scale-inv-Lipschitz}
Say that a domain $\Omega$ is a Lipschitz domain with scale invariant constants $(r_0,M)$ if for each $x_0 \in \partial \Omega$ there is a direction $\nu \in S^{d-1}$ and a function $f : \nu^{\perp} \to \R$ with $f(0) =0$, $\|Df\|_{\infty} \leq M$, and
\[ \Omega \cap B_{r_0|\Omega|^{1/d}}(x_0) = \{ x \in B_{r_0|\Omega|^{1/d}}(x_0): (x-x_0) \cdot \nu \leq f(P_{\nu^{\perp}}(x-x_0))\}. \]
In other words $\tilde{\Omega}$ is an $(r_0,M)$ Lipschitz domain.
\end{definition}

\begin{definition}\label{d.large-scale-lipschitz}
Say that a domain $\Omega$ is a (large scale) Lipschitz domain with constants $(r_0,M,h)$ if there is an $(r_0,M)$ Lipschitz domain $\Omega_- \subset \Omega$ with
\[d_H(\partial \Omega_-,\partial \Omega) \leq h. \]
\end{definition}
Note that the parameter $h$ is not scale invariant.  If $\Omega$ is a $(r_0,M,h)$ (large scale) Lipschitz domain then $\tilde{\Omega}$ is a $(r_0,M,h|\Omega|^{-1/d})$ (large scale) Lipschitz domain.  The length scale $h$ will be related to the length scale of the oscillations in the operator $a(x)$ which, for us, is normalized to $1$.  We also remark at this point that when we say ``large scale" Lipschitz regularity we really mean ``at all scales above $1$".  

\begin{remark}We remark that \dref{scale-inv-Lipschitz} implies a scale invariant perimeter bound, \dref{perimeter-hyp}, and a scale invariant inner/outer density bound, \dref{scale-inv-densities}, with constants depending on $(r_0,M)$.  The point is that \dref{perimeter-hyp}, \dref{scale-inv-densities}, and \dref{weakly-regular} can be proven earlier in the line of arguments and can be used to establish sub-optimal quantitative homogenization rates in the domain which are then needed to get a Lipschitz estimate.
\end{remark}

\subsection{Regularity of one-phase almost minimizers}\label{s.almost-min-reg-review} In this section we will recall several results from the literature on the regularity theory of almost minimizers of one-phase Bernoulli-type functionals.  Consider the energy functional
\[ \mathcal{J}_Q(w,D) = \int_{D} a(x) \grad w \cdot \grad w + Q(x)^2{\bf 1}_{\{w>0\}} \ dx\]
where $Q(x)$ is (for now) bounded measurable
\[ (1+\gamma)^{-1} \leq Q(x) \leq 1+\gamma.\]
We give several results about almost minimizers of $\mathcal{J}$.  Most of the results below are from the previous work of the author \cite{FeldmanReg} which adapted ideas by De Silva and Savin \cite{DeSilvaSavin} to the periodic oscillatory operator / large scale case.

First is the large scale Lipschitz estimate.  We mention, for context, that small scale Lipschitz estimates for almost or quasi-minimizers have been previously established for  Bernoulli-type starting with work of David and Toro \cite{DavidToro} and with further developments in \cite{DavidEngelsteinToro,DavidEngelsteinSVGToro,Trey}.

\begin{theorem}[Version of F. \cite{FeldmanReg} Theorem 1]\label{t.almostmin-lip-est}
Suppose that $1 \leq R \leq R_0$ and $w \in H^1(B_R(0))$ non-negative satisfies, for all $0 \leq r \leq R$,
\[ \mathcal{J}_{Q}(w,B_r(0)) \leq (1+(r/R_0))\mathcal{J}_{Q}(v,B_r(0))+(r/R_0) |B_r|\]
for $v \in w + H^1_0(B_r(0))$ the $a$-harmonic replacement.  Then for all $1 \leq r \leq R$
\[ \|\grad w\|_{{\underline{L}}^2(B_r(0))} \leq C(d,\Lambda,\gamma)(1+ \|\grad w\|_{\underline{L}^2(B_R(0))}).\]
If, additionally, $a \in C^{0,1}$ then
\[ |\grad w(0)| \leq C(d,\Lambda,\gamma,\|\grad a\|_\infty)(1+ \|\grad w\|_{\underline{L}^2(B_R(0))}).\]
\end{theorem}
With the Lipschitz estimate in hand we can proceed to several important notions of weak free boundary regularity. Non-degeneracy of the eigenfunction and density estimates and Hausdorff $(d-1)$-dimensionality of the free boundary. The following is a combination of all the results in Section 3 and Appendix A.5 of \cite{FeldmanReg}.  
\begin{proposition}\label{p.AM-misc-reg}
Suppose $w\in H^1(D)$ has $\|\grad w\|_{L^\infty(D)} \leq L$ and
\[\mathcal{J}_Q(w,B_r) \leq \mathcal{J}_Q(v,B_r)+\sigma |B_r|\]
for all $v \in w + H^1_0(B_r)$ and all $B_r \subset D$. Then there is $\sigma_0>0$, depending on $L$ and universal parameters so that if $\sigma \leq \sigma_0$ then:
\begin{enumerate}[label = (\roman*)]
\item (Weak non-degeneracy) There is $c(\Lambda,d,\|\grad a\|_\infty)>0$ so that if $w>0$ in $B_r$ then
\[ w(0) \geq c r.\]
\item (Strong non-degeneracy) There is $c(\Lambda,d,\|\grad a\|_\infty,L)>0$ so that if $x \in \partial \{w>0\}$ and $r>0$
\[ \sup_{y \in B_r(x)} w(y) \geq c r.\]
\item (Inner and outer density estimate) There is $c(\Lambda,d,\|\grad a\|_\infty,L)>0$ so that if $x \in \partial \{w>0\}$ and $r>0$
\[ c \leq \frac{|B_r(x) \cap \{w>0\}|}{|B_r(x)|} \leq 1-c.\]
\item (Boundary strip area estimate)  There is $C(\Lambda,d,\|\grad a\|_\infty,L)\geq 1$ so that if $B_{2r}(x_0) \subset D$ then
\[ |\{x \in B_r(x_0): d(x,\partial U) \leq t\}| \leq C(\sigma+\frac{t}{r})|B_r|.\]
\end{enumerate}
\end{proposition}

Finally one can obtain Lipschitz regularity of the free boundary down to the unit scale if the solution is sufficiently flat at a large scale.    Note that $C^{1,\alpha}$ regularity is also obtained at intermediate scales, but we don't use that. 

\begin{theorem}[F. \cite{FeldmanReg} Theorem 2]\label{t.REGmain2}
Let $\omega$ a modulus of continuity so that $\omega^{\frac{1}{d+4}}$ is a Dini modulus, $w \in H^1(B_{2R}(0))$ non-negative, $\|\grad w\|_{L^\infty(B_{2R})} \leq L$, satisfying the almost minimality conditions: 
\begin{enumerate}[label = (\roman*)]
\item There is $\sigma \leq \sigma_0(L,d,\Lambda,\|\grad a\|_\infty)$ from \pref{AM-misc-reg} so that for all $B_\rho \subset B_{2R}$,
\[ \mathcal{J}_{Q}(w,B_\rho) \leq \mathcal{J}_{Q}(v,B_\rho)+\sigma |B_\rho| \ \hbox{ for any } \ v \in w+H^1_0(B_\rho).\]
\item There is $R_0\geq1$ so that for all $0 \leq r \leq 2R$
\[ \mathcal{J}_{Q(0)}(w,B_r(0)) \leq \mathcal{J}_{Q(0)}(v,B_r(0))+\omega(r/R_0) |B_r| \ \hbox{ for any } \ v \in w+H^1_0(B_r).\]
\end{enumerate}
  Call $\tilde{\omega}(s) = \int_0^s \omega(t)^{\frac{1}{d+4}}\frac{dt}{t}$, a modulus of continuity since $\omega^{\frac{1}{d+4}}$ is Dini.  There are constants $\delta_0>0$, $c>0$ and $C \geq 1$ depending on $(d,\Lambda,L,\|\grad a\|_\infty)$ so that if $\tilde{\omega}(R/R_0) \leq c$, $0 \in \partial \{w>0\}$ and
\[ \inf_{\nu \in S^{d-1}}\sup_{x \in B_R} \frac{1}{R}|w(x) - \frac{1}{(\nu \cdot \bar{a} \nu)^{1/2}} (x \cdot \nu)_+| \leq \delta  \]
then, for all $1 \leq r \leq R$ there is $\nu' \in S^{d-1}$ such that
\[ \sup_{x \in B_r(y)} \frac{1}{r}|w(x) - \frac{1}{(\nu' \cdot \bar{a} \nu')^{1/2}} (x \cdot \nu')_+| \leq C[\tilde{\omega}(r/R_0)+r^{-\ep}]\]
where $\ep(d)>0$, and also
\[ |\nu'-\nu| \leq C[\delta + \tilde{\omega}(R/R_0)+r^{-\ep}].\]
\end{theorem}

\subsection{Optimal Faber-Krahn stability}
In order to show convergence of the optimizing sets from convergence of the energies we will take advantage of the optimal quantitative stability of Faber-Krahn type inequalities proved by Brasco, De Philippis and Velichkov \cite{BDPV}.  We state a restricted version of their full theorem here.

\begin{theorem}[Brasco, De Philippis, Velichkov \cite{BDPV}]\label{t.FaberKrahn}
There exists a positive constant $c_d$ depending only on dimension such that for every open set $\Omega \subset \R^d$ with finite measure and any ball $B$,
\[ |\Omega|^{\frac{2}{d}}\lambda_1(\Omega,\textup{id}) - |B|^{\frac{2}{d}}\lambda_1(B,\textup{id}) \geq c_d \mathcal{A}(\Omega)^2,\]
where $\mathcal{A}(\Omega): = \inf_B |\Omega \Delta B|/|B|$ is the Fraenkel asymmetry.
\end{theorem}

Note that this is a theorem about constant coefficient operators.  Parts of our results \tref{main} could be viewed as extensions of this result to the homogenization setting, although, of course, we are applying not giving a new proof of their theorem.
\begin{remark}\label{r.faber-krahn-remark}
Note that
\[ ||\Omega| - |B|| \leq |\Omega \Delta B|.\]
So if $B$ achieves the infimum in the Fraenkel assymetry then and $t$ is a dilation such that $|tB| = |\Omega|$ then
\[   |t^d - 1||B| = |tB \Delta B| = ||tB| - |B|| =  ||\Omega| - |B|| \leq |\Omega \Delta B|\]
and so
\[ \frac{|\Omega \Delta tB|}{|tB|} \leq \frac{|B|}{|tB|}\left[\frac{|\Omega \Delta B|}{|B|} + \frac{|B \Delta tB|}{|B|}\right]\leq 2t^{-d}\frac{|\Omega \Delta B|}{|B|} \]
and if $|\Omega \Delta B|/|B| \leq \frac{1}{2}$ then $t^{-d} \leq 2^d$. So if $|\Omega|^{\frac{2}{d}}\lambda_1(\Omega,\textup{id}) - |B|^{\frac{2}{d}}\lambda_1(B,\textup{id}) \leq c(d)$ sufficiently small then the same inequality (up to changing the dimensional constant) holds with the asymmetry measure
\[ \inf_{|B| = |\Omega|} \frac{|B \Delta \Omega|}{|\Omega|}\]
which we find more convenient.
\end{remark}

\subsection{A key domain perturbation estimate}
In this section we give a key estimate on the variation of $\lambda_1(U,a)$ under inward perturbations when $U$ satisfies some very strong regularity properties adapted to the operator $a$.  Of course the point is these regularity properties will hold for domain minimizers.  
\begin{lemma}\label{l.interior-approx}
 Suppose that $U$ has $|U| \geq 1$ and satisfies the following properties
 \begin{enumerate}[label = (\roman*)]
 \item $-\grad \cdot( a(x) \grad \cdot)$-Eigenfunction $L$-Lipschitz estimate \dref{L-Lipschitz}.
  \item $-\grad \cdot( a(x) \grad \cdot)$-Eigenfunction $\ell$-non-degeneracy estimate \dref{ell-non-degen}.
 \item Domain $(r_0,M,h)$-(large scale) Lipschitz estimate \dref{large-scale-lipschitz}, and let $U_- \subset U$ be the $(r_0,M)$ Lipschitz domain with $d_H(\partial U_-,\partial U) \leq h$.
 \end{enumerate}
 Then there is $C\geq 1$ depending on universal parameters and on $(r_0,M,h,L,\ell)$ so that
 \[|U|^{\frac{2}{d}}(\lambda_1(U_-,a) -\lambda_1(U,a))\leq    C |U|^{-\frac{1}{d}}\]
 \end{lemma}
 We should emphasize that a Lipschitz domain property like (iii) is very much \emph{not} sufficient to imply properties (i) and (ii) above.  Those properties will come from $U$ being a domain minimizer of an augmented functional later.
 \begin{proof}
 Let $u_U$ be, as usual, the principal $- \grad \cdot (a(x) \grad \cdot)$ eigenfunction for $U$ with $\|\grad u_U\|_\infty \leq L |U|^{-\frac{1}{2}-\frac{1}{d}}$ by assumption.  Note then that
\[w = (u_U - |U|^{-\frac{1}{2}-\frac{1}{d}}Lh)_+ \in H^1_0(U_-).\] By non-degeneracy assumption 
\[ u_U \geq \ell |U|^{-\frac{1}{2}-\frac{1}{d}} d(x,\R^d \setminus U)\] 
so $w > 0$ in $V_- = \{x \in U_- : d(x,\partial U_-) \geq \frac{L}{\ell}h\}$.

So the $L^2$ norm of $w$ can be estimated from below
\begin{align*}
 \|w\|_{L^2} &\geq \|u_U {\bf 1}_{V_-}\|_{L^2} - Lh|U|^{-\frac{1}{d}} \\
 &\geq 1-Lh|U|^{-\frac{1}{d}} - \|u_U {\bf 1}_{U \setminus V_-}\|_{L^2}\\
 &\geq 1-Lh|U|^{-\frac{1}{d}} - L|U|^{-\frac{1}{2} - \frac{1}{d}}(1+\frac{L}{\ell})h|U \setminus V_-|^{\frac{1}{2}}\\
 &\geq 1-Lh|U|^{-\frac{1}{d}} - C(L,\ell,r_0,M,h)|U|^{-\frac{1}{2} - \frac{1}{d}+\frac{d-1}{2d}}\\
 &\geq 1-C|U|^{-\frac{1}{d}} - C|U|^{-\frac{3}{2d}}
 \end{align*}
using in the last two lines that $u_U \leq \|\grad u_U\|_\infty(1+\frac{L}{\ell})h$ in $U \setminus V_-$ and $|U \setminus V_-| \leq C(r_0,M)(1+\frac{L}{\ell})h|U|^{\frac{d-1}{d}}$. The second term error term is lower order when $|U| \geq 1$ so this simplifies to
\[ \|w\|_{L^2}^2 \geq 1 - C(L,\ell,h,r_0,M) |U|^{- \frac{1}{d}}.\]
Then
\begin{align*}
 \lambda_1(U,a) &= \int_{U} a(x)\grad u_U \cdot \grad u_U \ dx \\
 &\geq \int_{U_-} a(x)\grad w \cdot \grad w \ dx\\
 &\geq (1 - C|U|^{ - \frac{1}{d}})\lambda_1(U_-,a)\\
 &\geq \lambda_1(U_-,a) - C|U|^{ - \frac{1}{d}}|U|^{-\frac{2}{d}}.
 \end{align*}
 Rearranging gives the result.
 \end{proof}

\section{Results from the $L^2$-theory of Dirichlet homogenization}\label{s.L2theory}

In this section we collect some known results on the  quantitative $L^2$-theory for homogenization of Dirichlet data problems in divergence form.  We will also give some minor variations which are needed for our problem.  The reader could potentially skip this section on initial reading, proceed to \sref{augmented-reg} and refer back here as needed.

For a bounded domain $U$ and $f \in L^2(U)$ consider the variational solutions $v,v_0\in H^1_0(U)$ of the problems
\begin{equation}
\begin{cases}
-\grad \cdot (a (x)\grad v)= f & \hbox{ in } \ U \\
v = 0 & \hbox{ in } \ U
\end{cases}
\ \hbox{ and } \ \begin{cases}
-\grad \cdot (\bar{a}\grad v_0)= f & \hbox{ in } \ U \\
v_0 = 0 & \hbox{ in } \ U.
\end{cases}
\end{equation}
If $|U|$ is large and has some scale invariant regularity we expect that $v$ is well approximated by $v_0$ in the $L^2$ sense.  The quantification of this statement depending on the domain regularity is the topic of this section.

We recall here the correctors $\chi_q$ for each $q \in \R^d$.  The corrector $\chi_q$ is the unique global $\Z^d$ periodic and mean zero solution of the problem
\begin{equation}\label{e.corrector-def}
 - \grad \cdot (a(x)(q+\grad \chi_q)) = 0  \ \hbox{ in } \ \R^d.
 \end{equation}
  It follows from the previous properties that $\chi_q$ depends linearly on $q \in \R^d$.  

The homogenized matrix is defined
\begin{equation}\label{e.baradef}
 \bar{a}_{ij} =\langle (e_i + \grad \chi_{e_i}(x))a(x)(e_j + \grad \chi_{e_j}(x)) \rangle 
 \end{equation}
where $e_i$ are the standard basis vectors.

\subsection{Sub-optimal $L^2$ convergence rate in weakly regular domains}

First we recall some weak assumptions on the domain $U \subset \R^n$ which allow for a H\"older rate of homogenization.  The key quantity to control, related to the boundary condition, is the energy in a boundary layer $U\setminus U_t$ where
\[ U_t = \{x \in U: d(x,\R^d \setminus U) > t\}.\]
These boundary layer terms can be controlled via a global Meyer's estimate for solutions of divergence form uniformly elliptic Dirichlet boundary value problems.  We provide an overview of the results and details can be found in Appendix C of the book of Armstrong, Kuusi and Mourrat \cite{AKM}.  In \cite{AKM,AShom} the domain is assumed to be Lipschitz but, at least as far as our purposes are concerned (see remark below), this is only used in order to get a bound on the inner/outer density quantity 
\begin{equation*}\label{e.kappadef2}
 \kappa_U = \sup_{z \in U, r >0} \frac{|B_r(z) \cap U|}{|B_{r/2}(z) \cap U|} + \sup_{z \in \partial U, r>0} \frac{|B_r|}{|B_r(z) \setminus U|}
 \end{equation*}
which we defined and discussed previous in \dref{scale-inv-densities}

\begin{proposition}[Consequence of Meyer's Estimate]
Let $U \subset \R^d$ be a bounded domain with $|U| = 1$, $\kappa_U < + \infty$, and $a(x)$ a uniformly elliptic coefficient matrix $\Lambda^{-1} \leq a \leq \Lambda$.  Suppose that $f \in L^2(U)$, $u \in  H^1_0(U)$ is the variational solution of
\[ \begin{cases}
- \grad \cdot (a(x) \grad v) = f & \hbox{in } U,\\
v = 0 & \hbox{on } \partial U.
\end{cases}\]
There exists $\delta(d,\Lambda,\kappa_U)>0$ and $C(d,\Lambda,\kappa_U)<+\infty$ so that $v \in W^{1,2+\delta}(U)$ and
\[ \|\grad v\|_{L^{{2+\delta}}(U)} \leq C\|f\|_{L^{2}(U)}.\]
\end{proposition}
\begin{remark}
Besides bounding $\kappa_U$, the Lipschitz domain property is also used in \cite{AKM} to write the right hand side $h \in W^{-1,p}$ of the equation $- \grad \cdot (a(x)\grad u) = h$ as $h = \grad \cdot H$ for a vector field $H \in L^p$ by solving a Poisson equation $- \Delta w = h$ in $U$ for $w \in W^{1,p}_0(U)$.  We have a stronger hypothesis on the right hand side $f \in L^2(U)$ so for our purposes we can just take
\[ H = \int_{U} \grad \Phi(x-y)f(y) \ dy \ \hbox{ where } \ \Phi \ \hbox{ is the Newtonian potential.} \]
Then by Sobolev embedding and the classical Calderon-Zygmund estimate
\[ \|H\|_{L^{\frac{2d}{d-2}}(U)} \leq \|H\|_{L^{\frac{2d}{d-2}}(\R^d)} \leq C\|D^2H\|_{L^2(\R^d)} \leq  C\|f\|_{L^2(U)}.\]
\end{remark}

Now from Meyer's estimate and H\"older's inequality we can derive
\[ \|\grad v\|_{L^2(U \setminus U_t)} \leq \|\grad v\|_{L^{2+\delta}(U)}|U \setminus U_t|^{\frac{\delta}{2+\delta}}\leq C\|f\|_{L^{2}(U)}|U \setminus U_t|^{\frac{\delta}{2+\delta}}.\]
The continuity of $|U \setminus U_t|$ at $t=0$ will also come into the quantitative homogenization estimate.  For the domain arising from the free boundary problem we consider this will be controlled via a perimeter type estimate as in \dref{perimeter-hyp}.

\begin{theorem}[Sub-optimal quantitative homogenization]\label{t.quanthom}
Suppose that $U$ is a weakly regular bounded domain, as in \dref{weakly-regular}. Call $r = |U|^{1/d}$, $f \in L^2(U)$ and $v,v_0 \in H^1_0(U)$ respectively variational solutions of
\[ \begin{cases}
- \grad \cdot (a(x) \grad v) = f & \hbox{in } U,\\
v = 0 & \hbox{on } \partial U
\end{cases} \ \hbox{ and } \ \begin{cases}
- \grad \cdot (\bar{a}\grad v_0) = f & \hbox{in } U,\\
v_0 = 0 & \hbox{on } \partial U.
\end{cases}\]
 Then there are $\alpha \in (0,1)$ and $C \geq 1$ depending on $d$, $\Lambda$, $\kappa_U$, and $P$ so that
\[ r^{-2} \| v -  v_0\|_{L^2(U)} \leq Cr^{-\alpha}\|f\|_{L^2(U)}.\]
\end{theorem}

We refer to \cite{ArmstrongSmart} for the proof, although there is some work in verifying that the constants depend on the parameters we have defined.

\subsection{Optimal $L^2$ convergence rate in Lipschitz domains}

In Lipschitz domains the $L^2$-theory can be pushed much further, almost to a linear rate, thanks to the deep works of Kenig, Lin, and Shen \cite{KenigLinShen}.

\begin{theorem}[Optimal $L^2$ homogenization in Lipschitz domains \cite{KenigLinShen}]\label{t.optimal-hom-lip}
Suppose that $U$ is a bounded $(r_0,M)$-Lipschitz domain in the sense of \dref{scale-inv-Lipschitz}, call $r = |U|^{1/d}$. Let $f \in L^2(U)$ and consider $v,v_0 \in H^1_0(U)$ respectively variational solutions of
\[ \begin{cases}
- \grad \cdot (a(x) \grad v) = f & \hbox{in } U,\\
v = 0 & \hbox{on } \partial U
\end{cases} \ \hbox{ and } \ \begin{cases}
- \grad \cdot (\bar{a}\grad v_0) = f & \hbox{in } U,\\
v_0 = 0 & \hbox{on } \partial U.
\end{cases}\]
 For any $\eta>0$ there is $C \geq 1$ depending on universal constants, $\eta$, and the  Lipschitz domain property so that
\[ r^{-2}\|v - v_0\|_{L^2(U)} \leq Cr^{-1}|\log(2+ r)|^{\frac{1}{2}+\eta} \|f\|_{L^2(U)}.\]

\end{theorem}
\subsection{Eigenvalue convergence}
 The Dirichlet $L^2$ theory above also allows to discuss the convergence of the Dirichlet eigenvalues in homogenization.  Given $f \in L^2(U)$ let $v \in H^1_0(U)$ be the solution of
 \begin{equation}
\begin{cases}
-\grad \cdot (a (x)\grad v)= f & \hbox{ in } \ U \\
v = 0 & \hbox{on } \ \partial U.
\end{cases}
\end{equation}
Define the Dirichlet solution operator
\[ T(a) f = v.\]
Let $T(\bar{a})$ be the corresponding solution operator for the homogenized operator $- \grad \cdot (\bar{a} \grad \cdot)$.

The following result is purely functional analysis from the min-max principle for eigenvalues, see Shen \cite[Lemma 7.2.1]{ShenBook}.
\begin{lemma}\label{l.eigenvalue-operator-est}
For all $k \geq 1$
\[|\lambda_k(U,a)^{-1} - \lambda_k(U,\bar{a})^{-1}| \leq \|T(a) - T(\bar{a})\|_{L^2 \to L^2}.\]
\end{lemma}

In particular we have the following Corollaries.  For weakly regular domains following from \tref{quanthom} and \lref{eigenvalue-operator-est}:
\begin{corollary}\label{c.suboptimal-eigen-conv}
Let $U$ a weakly regular domain, call $r = |U|^{1/d}$. There are $\alpha \in (0,1)$ and $C \geq 1$ depending on $d$, $\Lambda$, $\kappa_U$, and $P$ so that
\[ r^{-2}|\lambda_1(U,a)^{-1} - \lambda_1(U,\bar{a})^{-1}| \leq Cr^{-\alpha}. \]
\end{corollary}
For Lipschitz domains, by applying \tref{optimal-hom-lip} and \lref{eigenvalue-operator-est}:
\begin{corollary}\label{c.optimal-eigen-conv}
Let $U$ be a $(r_0,M)$-Lipschitz domain, call $r = |U|^{1/d}$.  For any $\eta>0$ there is $C \geq 1$ depending on universal constants, $\eta$, and the Lipschitz property so that
\[ r^{-2}|\lambda_1(U,a)^{-1} - \lambda_1(U,\bar{a})^{-1}| \leq Cr^{-1}|\log (2+r)|^{\frac{1}{2}+\eta}. \]
\end{corollary}

Finally we discuss domains which are (large scale) Lipschitz in the sense of \dref{large-scale-lipschitz}. This is a non-trivial issue because the continuity of the eigenvalues with respect to domain variations is, in general, very delicate.

  \begin{corollary}\label{c.quanthom-large-scale-lip}
Let $U$ satisfying the same hypotheses as in \lref{interior-approx}.  For any $\eta>0$ there is $C\geq 1$ depending on universal parameters and on $(r_0,M,h,L,\ell,\eta)$ so that
 \[|U|^{\frac{2}{d}}(\lambda_1(U,\bar{a}) -\lambda_1(U,a))_+ \leq  C |U|^{-\frac{1}{d}}|\log (2+|U|)|^{\frac{1}{2}+\eta}.\]
 \end{corollary}

\begin{proof}
By the hypothesis there is an $(r_0,M)$ - Lipschitz domain $U_- \subset U$ such that $d_H(\partial U_-,\partial U) \leq h$. Using domain monotonicity of eigenvalues and \lref{interior-approx}
\begin{align*}
 |U|^{\frac{2}{d}}(\lambda_1(U,\bar{a}) -\lambda_1(U,a)) &\leq |U|^{\frac{2}{d}}(\lambda_1(U_-,\bar{a}) - \lambda_1(U,a)) \\
 &\leq |U|^{\frac{2}{d}}(\lambda_1(U_-,\bar{a}) - \lambda_1(U_-,a)) + C|U|^{-\frac{1}{d}}
 \end{align*}
Finally since $U_-$ is a Lipschitz domain we can apply \cref{optimal-eigen-conv} to estimate the first term on the right and get the result.
\end{proof}

\subsection{Principal eigenfunction quantitative convergence}The aim of this section is to obtain a (suboptimal) quantitative error estimate in $L^\infty$ for the principal eigenfunctions. This estimate plays an essential role in \sref{augmented-reg} establishing initial flatness of the eigenfunctions associated with $J_g$ domain minimizers, see \pref{generic-Jg-errorest}.

We will prove:
\begin{proposition}\label{p.eigenfunction-conv}
Suppose that $U$ is a bounded domain so that the eigenfunction $u_U$ has the (scaled) $L$-Lipschitz property \dref{L-Lipschitz} then there is $C(\Lambda,d,L)$ so that
\begin{align*} &|U|^{\frac{1}{2}}\|u_U(\cdot,a) - u_E(\cdot,\bar{a})\|_{L^\infty}   \\
&\quad \quad \quad \leq\bigg||U|^{\frac{2}{d}}(\lambda_1(U,a) - \lambda_1(E,\bar{a}))\bigg|^{\frac{1}{d+2}}+\left(\frac{d_H(\partial U,\partial E)}{|U|^{\frac{1}{d}}}\right)^{\frac{1}{d+2}} +|U|^{-\frac{1}{2d(d+2)}}
\end{align*}
\end{proposition}
We proceed in a series of Lemmas culminating in a proof of \pref{eigenfunction-conv}.  The idea is standard but we didn't find an applicable reference.  Typically estimates of eigenfunctions are harder than estimates of eigenvalues because of possibly degeneracy.  However we are studying the principal eigenvalue, which is not degenerate, in a scenario where the domain $U$ is close to an ellipsoid $E=\bar{a}^{\frac{1}{2}}B$. So we can take advantage of a universal lower bound on the spectral gap $|E|^{\frac{2}{d}}(\lambda_2(E,\bar{a}) - \lambda_1(E,\bar{a})) \geq c $ over the ellipsoid class.

\begin{lemma}\label{l.ground-state-close}
Suppose that $a(x)$ satisfies \aref{a1} and $v \in H^1_0(\Omega)$ with $\|v\|_{L^2(\Omega)} = 1$ satisfies
\[ \int_{\Omega} a(x) \grad v \cdot \grad v  \ dx \leq (1+\delta)\lambda_1(\Omega,a).\]
Then
\[ (\tfrac{\lambda_2(\Omega,a)}{ \lambda_1(\Omega,a)} -1)\min_{s \in \pm 1}\|v -  s u_\Omega\|_{L^2(\Omega)}^2 \leq 4\delta\]
\end{lemma}
\begin{proof}
For brevity we write $\lambda_j = \lambda_j(\Omega,a)$.  Note, from the standard characterization of $\lambda_2$,
\[ \lambda_1|(v,u_\Omega)|^2+\lambda_2\|v - (v,u_\Omega)u_\Omega\|_{L^2(\Omega)}^2 \leq \int_{\Omega} a(x)\grad v \cdot \grad v \ dx \leq \lambda_1\|v\|_{L^2(\Omega)}^2+\delta\lambda_1 \]
and also note that 
\[ 1= \|v\|_{L^2(\Omega)}^2 = |(v,u_\Omega)|^2+\|v - (v,u_\Omega)u_\Omega\|_{L^2(\Omega)}^2.\]
So, rearranging the earlier inequality,
 \[(\tfrac{\lambda_2}{\lambda_1}-1)\|v - (v,u_\Omega)u_\Omega\|_{L^2(\Omega)}^2\leq \delta\]
 and also
 \[ (\tfrac{\lambda_2}{\lambda_1}-1)\left[1-|(v,u_\Omega)|^2\right] \leq \delta.\]
Now assume, without loss, that $(v,u_\Omega) \geq 0$, then
 \[ (\tfrac{\lambda_2}{\lambda_1}-1)^{1/2}\| v - u_\Omega\| \leq \delta^{1/2} + (\tfrac{\lambda_2}{\lambda_1}-1)^{1/2}|1 - (v,u_\Omega)|\]
 and by the assumption $(v,u_\Omega) \geq 0$ we have $|1 - (v,u_\Omega)| \leq 1-|(v,u_\Omega)|^2$ so
\[ (\tfrac{\lambda_2}{\lambda_1}-1)^{1/2}\| v - u_\Omega\| \leq 2 \delta^{1/2}.\]
\end{proof}

In the remainder of the section we will abuse notation a bit and denote $u_U = u_U(\cdot,a)$, $u_E = u_E(\cdot,a)$ and $\bar{u}_E = u_E(\cdot,\bar{a})$.

By \lref{eigenvalue-operator-est} and \tref{optimal-hom-lip}
\[ |\tfrac{\lambda_2(E,a)}{\lambda_1(E,a)} - \tfrac{\lambda_2(E,\bar{a})}{\lambda_1(E,\bar{a})}| \leq C|E|^{-\frac{1}{d}}\]
where $C$ is universal. Also 
\[ \tfrac{\lambda_2(E,\bar{a})}{\lambda_1(E,\bar{a})} = \tfrac{\lambda_2(E_1,\bar{a})}{\lambda_1(E_1,\bar{a})} =\tfrac{\lambda_2(B_1,\textup{id})}{\lambda_1(B_1,\textup{id})} >1\]
so there is a universal $C_0 \geq 1$ so that $|E| \geq C_0$ implies that $\frac{\lambda_2(E,a)}{\lambda_1(E,a)}-1 \geq c_0$ universal.

Let $\zeta$ be a smooth cutoff function $0 \leq \zeta \leq 1$ which is $1$ in $(1-t)B$ and $0$ outside $B$ and $|\grad \zeta| \leq Ct^{-1}$
\[ \xi(x) = \bar{u}_E(x) + \zeta(x) \grad \bar{u}_E(x) \cdot \chi(x). \] 
Here $\chi = [\chi_{e_1},\dots,\chi_{e_d}]$ is the corrector matrix defined in \eref{corrector-def}. If we take $t = |E|^{\frac{1}{2d}}$ then we get the statement:
\begin{lemma}\label{l.quant-hom-intermediate}
Let $\xi$ as above then
\[ \int_{E} a(x) \grad \xi \cdot \grad \xi  \ dx \leq (1+C|E|^{-\frac{1}{2d}})\lambda_1(E,\bar{a}).\]
\end{lemma}
This follows from \cite{FeldmanReg}[Proposition 23] after appropriate rescaling.
\begin{lemma}\label{l.uE-baruE}
If $|E| \geq C_0$ universal then
\[ \|\bar{u}_E- u_E\|_{L^2(E)} \leq C(\Lambda,d) |E|^{-\frac{1}{4d}}.\]
\end{lemma}
\begin{proof}
Applying \lref{ground-state-close} to $\xi$ we find, for $|E| \geq C_0$,
\[
c_0\|\xi - u_E\|_{L^2(E)}^2 \leq C |E|^{-\frac{1}{2d}}
\]
where we also used, by \lref{eigenvalue-operator-est} and \tref{optimal-hom-lip},
\[ \lambda_1(E,a)^{-1}(\lambda_1(E,\bar{a})-\lambda_1(E,a)) \leq C |E|^{-\frac{1}{d}}.\]
Then returning to $\bar{u}_E$ we get
\begin{equation}
c_0^{1/2}\|\bar{u}_E- u_E\|_{L^2(E)} \leq C |E|^{-\frac{1}{4d}}
\end{equation}
where we are using $\|\grad \bar{u}_E\|_\infty \leq C|E|^{-\frac{1}{2}-\frac{1}{d}}$ so that $\|\xi - \bar{u}_E\|_{L^2} \leq C|E|^{\frac{1}{2}}.|E|^{-\frac{1}{2}-\frac{1}{d}}$.
\end{proof}

Next we estimate $u_U - u_E$ in $L^2$.

\begin{lemma}\label{l.L2close-UE}
There is $C_0$ universal so that if $|U| \geq C_0$ and $E$ is a $\bar{a}$-ellipsoid with $2E \supset U$ then
\[ \|u_U - \bar{u}_E\|_{L^2} \leq C(\Lambda,d) \left[|U|^{\frac{1}{d}}|\lambda_1(U,a) - \lambda_1(E,\bar{a})|^{1/2}+\left(\frac{d_H(\partial E,\partial U)}{|U|^{\frac{1}{d}}}\right)^{1/2}+|U|^{-\frac{1}{4d}}\right].\]
\end{lemma}

\begin{proof}
By replacing with a dilation $(1+t)E$ for some 
\[ 0 \leq t  \leq C(\Lambda,d)|E|^{-\frac{1}{d}}d_H(\partial E,\partial U)\]
we can assume, without loss, that $E \supset U$.  Note that the explicit $L^2$ continuity of $\bar{u}_{(1+t)E}$ in the $t$ variable follows from the dilation estimate for $0 \leq t \leq 1$
\[ \|\bar{u}_{(1+t)E} - \bar{u}_E\|_{L^2(E)}^2 = \|(1+t)^{-d/2}\bar{u}_{E}((1+t)^{-1}\cdot) - \bar{u}_E\|_{L^2(E)}^2\leq Ct^2|E|^{\frac{2}{d}}\|\grad \bar{u}_E\|_{L^2(E)}^2.\]
And for the eigenvalue
\[\lambda_1(E,\bar{a}) - \lambda_1((1+t)E,\bar{a}) =(1-(1+t)^{-2})\lambda_1(E,\bar{a}) \leq Ct \lambda_1(E,a).\]

Now, with the assumption $E \supset U$ in effect,
\[ \int_{E} a(x) \grad u_U\cdot \grad u_U \ dx = \lambda_1(U,a) \leq \lambda_1(E,a) + (\lambda_1(U,a) - \lambda_1(E,a)).\]
So we apply \lref{ground-state-close} to get
\[c_0\|u_U - u_E\|_{L^2(E)}^2 \leq 4\lambda_1(E,a)^{-1}(\lambda_1(U,a) - \lambda_1(E,a)).\]
And use
\[ \lambda_1(E,a)^{-1}(\lambda_1(E,\bar{a}) - \lambda_1(E,a)) \leq C|E|^{-\frac{1}{d}}.\]
By the previous \lref{uE-baruE}
\[ \|u_E - \bar{u}_E\|_{L^2(E)} \leq C(\Lambda,d) |E|^{-\frac{1}{4d}} \]
so we can conclude with a triangle inequality.

\end{proof}

\begin{proof}[Proof of \pref{eigenfunction-conv}]
By assumption $|U|^{\frac{1}{2}+\frac{1}{d}}\|\grad u_U\|_{L^\infty} \leq L$ and by the explicit form of $\bar{u}_E$ as a linear transformation of the principal eigenfunction of $-\Delta$ on $B_1$ also $|E|^{\frac{1}{2}+\frac{1}{d}}\|\grad \bar{u}_E\|_{L^\infty} \leq C(\Lambda,d)$.  Therefore 
\[ \|u_U - \bar{u}_E\|_{L^\infty} \leq C\|\grad (u_U - \bar{u}_E)\|_{L^\infty}^{\frac{d}{d+2}}\|u_U - \bar{u}_E\|_{L^2}^{\frac{2}{d+2}}.\]
So
\[ |U|^{\frac{1}{2}}\|u_U - \bar{u}_E\|_{L^\infty} \leq C(\Lambda,d,L)\|u_U - \bar{u}_E\|_{L^2}^{\frac{2}{d+2}}\]
and applying \lref{L2close-UE} finishes the proof.
\end{proof}

\section{Regularity theory of augmented functional minimizers}\label{s.augmented-reg}
In this main section of the paper we establish regularity and large volume asymptotics results for domain minimizers of the augmented functional
\begin{equation}\label{e.Jg-general-def}
 J_g(U,a) = \lambda_1(U,a) + \int_U g(x) dx.
 \end{equation}
We also recall the effective functional
\[ J_g(U,\bar{a}) = \lambda_1(U,\bar{a}) + \int_U g(x) dx.\]
The slight generalization to allow $x$-dependence in the volume term will be useful for application to the hard constraint problem in \sref{relation}.  

Our aim is to consider minimization of these functionals on $\R^d$, however the existence of minimizers does not follow from the a general theorem since the embedding $H^1(\R^d) \hookrightarrow L^2(\R^d)$ is not compact.  Thus we will need to explain existence during the course of the proof and so it will be useful to consider the same functionals also set on a large torus $N\T^d$, i.e. $\R^d / N\Z^d$ for some integer $N \gg 1$.

We make the following assumptions on $g:\R^d \to \R$ which is meant to be a perturbation of a constant value $\mu>0$: 
\begin{itemize}
\item (Scaled ellipticity) For some $\gamma \geq0$
\begin{equation}\label{e.g-hyp1}
 (1+\gamma)^{-1} \leq g(x)/\mu \leq 1+\gamma. 
 \end{equation}
\item (Scaled Dini continuity) There is a modulus of continuity $\omega_g(s) \geq |s|$ such that $\omega_g^{\frac{1}{d+4}}$ is a Dini modulus and
\begin{equation}\label{e.g-hyp2}
 |g(x) - g(y)| \leq \mu\omega_g (\mu^{\frac{1}{d+2}}|x-y|).
 \end{equation}
\item (Minimizer localizing property) There is an $R_g>0$ such that
\begin{equation}\label{e.g-hyp3}
  g(x) \equiv \sup_{\R^d} g \ \hbox{ for } \ |x| \geq R_g.
  \end{equation}
\end{itemize}
\begin{remark}\label{r.g-hyp4}
When we are considering minimization on a large torus $\R^d \bmod N\Z^d$ we will assume that $N \geq 2R_g$ and define $g$, abusing notation, to be the periodic extension of $g|_{[-N/2,N/2)^d}$.  Which still satisfies the first two assumptions because it is constant in a neighborhood of $\partial [-N/2,N/2)^d$ by the third assumption.
\end{remark}
\begin{remark}
The Dini modulus assumption is not just to show the extent of the method, we actually need it later in the paper.  Periodic oscillatory $g$ could be considered actually, as in \cite{FeldmanReg}, but that would be extraneous for the main goals of this paper.  The condition $\omega_g(s) \geq |s|$ is just for convenience so that additional linear error terms can be absorbed into $\omega_g$ as a ``worst case" modulus.  The power $\omega_g^{\frac{1}{d+4}}$ needing to be a Dini modulus is probably not sharp. The precise origin is the large scale viscosity solution property proved by De Silva and Savin \cite{DeSilvaSavin} for almost minimizers resulting in their Lemma 4.5. 
\end{remark}

For the full free boundary regularity, we will also need to assume $\gamma$ is close to $1$ so that the minimizers of the effective functional are close to ellipsoids and we can guarantee the free boundary regularity.

Our main result in this section is the following regularity theory:
\begin{theorem}\label{t.augmented-full}
Suppose $U$ minimizes $J_g$ from \eref{Jg-general-def} over quasi-open subset of $\R^d$ \emph{or} $N\T^d$ for some integer $N \geq R_g \vee \mu^{-\frac{1}{d+2}}$. Call $u_U$ to be the corresponding principal Dirichlet eigenfunction on $U$ of $-\grad \cdot (a(x) \grad\cdot)$. Then $|U| \sim_{d,\Lambda} \mu^{-\frac{d}{d+2}}$ and:
\begin{enumerate}[label = (\roman*)]
\item\label{part.1} The principal eigenfunction $u_U$ is Lipschitz with (scale invariant) bound
\[ |U|^{\frac{1}{2}+\frac{1}{d}}\|\grad u_U\|_{L^\infty} \leq C(\Lambda,d,\gamma,\omega_g,\|\grad a\|_\infty).\]
\item\label{part.2} The principal eigenfunction $u_U$ is non-degenerate with (scale invariant) bound
\[ |U|^{\frac{1}{2}+\frac{1}{d}}\sup_{y\in B_r(x)} \frac{u_U(y)}{r} \geq c(\Lambda,d,\gamma,\omega_g,\|\grad a\|_\infty) \] 
for all  $x \in \partial U$ and $0 < r \leq |U|^{\frac{1}{d}}$.
\item\label{part.4} There is universal $\gamma_0>0$ so that if $\gamma \leq \gamma_0$ then $U$ is a (large scale) Lipschitz domain in the sense of \dref{large-scale-lipschitz} with parameters $(r_0,1,h)$ and $(r_0,h)$ depend only on $(\Lambda,d,\omega_g,\|\grad a\|_\infty)$.
\item\label{part.5} Let $\gamma \leq \gamma_0$ as in \partref{4}.  For any $\eta>0$ there is $C(\Lambda,d,\omega_g,\|\grad a\|_\infty,\eta)\geq 1$ so that
\[ |U|^{\frac{2}{d}}|\lambda_1(U,a) - \inf_{|V| = |U|} \lambda_1(V,\bar{a})| \leq C\left[\gamma+|U|^{-\frac{1}{d}}|\log(1+|U|)|^{\frac{1}{2}+\eta}\right] \]
and
\[ |U|^{2/d}(\lambda_1(U,\bar{a}) - \lambda_1(U,a)) \leq  C\left[\gamma+|U|^{-\frac{1}{d}}|\log(1+|U|)|^{\frac{1}{2}+\eta}\right].\]
\end{enumerate}
\end{theorem}
Note that \tref{main-aug} follows directly from this statement taking $\gamma = 0$, i.e. $g \equiv \mu$.  Also note that we made the statement for all $\mu>0$, but parts \partref{4} and \partref{5} do not have interesting content when $\mu>0$ large and/or $|U|$ small.

\subsection{Outline of the arguments}
We structure the following section as an outline but this is also our presentation of the proof of \tref{augmented-full}.  Within the outline we will state numerous Lemmas and Propositions.  These will either be citations of results from the literature, or will require some proof which will then be provided in the latter parts of \sref{augmented-reg}.

We explain the sequence of regularity estimates which are needed. The first big goal is the (large scale) Lipschitz regularity of the eigenfunction, then the (large scale) Lipschitz regularity of the domain, and finally linear (with logarithmic factors) homogenization rates.  This will involve applications of regularity theory reviewed in \sref{almost-min-reg-review}. We will need sub-optimal quantitative homogenization estimates, recalled in \sref{L2theory}, in order to establish initial flatness needed for higher free boundary regularity.  Then with the higher regularity (Lipschitz) we can apply optimal quantitative homogenization estimates.

{\bf Step 1.} ($L^\infty$ estimate) Obtain an $L^\infty$ estimate of the eigenfunctions.  We can quote a result from \cite{RussTreyVelichkov} (Lemma 5.4) and Trey~\cite{Trey} (Lemma 2.1 and Lemma 3.2).

\begin{proposition}[\cite{RussTreyVelichkov}, \cite{Trey}]\label{p.linfty}
Suppose that $a(x)$ is a uniformly elliptic matrix field.  There is $n(d) \in \mathbb{N}$ and $C(d,\Lambda) \geq 1$ so that if $u_\Omega$ is the first Dirichlet eigenfunction for the operator $- \grad \cdot (a(x) \grad)$ on a bounded quasi-open set $\Omega$ of volume $|\Omega| = 1$ then
\[ \|u_\Omega\|_{L^\infty} \leq C\lambda_1(\Omega,a)^n\]
\end{proposition}
The proof is essentially to notice that the solution / resolvent operator $T(a,U) f$ which maps $f \in L^2(U)$ to the variational $H^1_0(U)$ solution of $- \grad \cdot (a(x) \grad v ) = f$, by Sobolev embedding, maps $T(a,U) : L^2(U) \to L^{p}(U)$ for $2 \leq p \leq \frac{2d}{d-2}$ (strict inequality in $d=2$).  On the other hand one can also prove that $T(a,U) : L^p(U) \to L^\infty(U)$ for $p > d/2$ (see \cite{Trey}[Lemma 2.1]).  So some power $T(a,U)^{n(d)}$ maps $L^2(U) \to L^\infty(U)$.

By rescaling a domain $\Omega$ to $\tilde{\Omega} = |\Omega|^{-\frac{1}{d}}\Omega$ we get, for a quasi-open $\Omega$
\begin{equation}
 |\Omega|^{\frac{1}{2}}\|u_{\Omega}\|_{L^\infty} \leq C\lambda_1(\tilde{\Omega},a)^n
 \end{equation}
 Combining this with \lref{scaling-bounds} we find:
 \begin{corollary}\label{c.linfty}
 Suppose that $a(x)$ is a uniformly elliptic matrix field and $U$ is a $J_g(\cdot,a)$ minimizer then
 \[ |U|^{1/2}\|u_{U}\|_{L^\infty} \leq C(\Lambda,d). \]
 \end{corollary}

{\bf Step 2.} (Almost minimality property) Next we show that the normalized eigenfunctions $w_U:= \mu^{-1/2}u_U$ associated with $J_g$ minimizing domains $U$ are almost minimizers of a Bernoulli-type energy functional
\[ \mathcal{J}_{g/\mu}(w,\Omega) = \int_{\Omega}\frac{1}{2}\grad w \cdot a(x) \ \grad w + \mu^{-1}g(x){\bf 1}_{\{w>0\}} dx.\]
 This is a standard idea in the study of eigenvalue shape optimization problems.  The error in the almost minimization property will depend on the $L^\infty$-norm of the eigenfunction which we have control on by \pref{linfty}.

\begin{lemma}\label{l.almostminprop}
Suppose that $U$ minimizes $J_g(\cdot,a)$, call $w_U = \mu^{-1/2}u_U$.  Call $R_0 = c(\lambda,d)|U|^{\frac{1}{d}}$ for a small enough $c(\Lambda,d)>0$. Then, for any $0 < r \leq R_0$ and $v \in w_U + H^1_0(B_r(x_0))$ the $a$-harmonic replacement of $w_U$ in $B_r$:

\[ \mathcal{J}_{g/\mu}(w_U,B_r) \leq \big(1+r/R_0\big)\mathcal{J}_{g/\mu}(v,B_r)+(r/R_0)|B_r|. \]  

If, furthermore, $w_U$ is $L$-Lipschitz then for {\bf any} $v \in w_U + H^1_{0}(B_r(x_0))$
\[ \mathcal{J}_{g/\mu}(w_U,B_r) \leq  \mathcal{J}_{g/\mu}(v,B_r)+(r/R_0)|B_r| \]
now for $R_0 = c(\Lambda,d,L)|U|^{\frac{1}{d}}$ with a sufficiently small $c(\Lambda,d,L)>0$.
\end{lemma}

The proof is postponed to \sref{almost-min-proof} below and follows the natural idea of rescaling to maintain the $L^2$ constraint after perturbation.  There are some tricky computational ideas to make sure there are no additional requirements on the perturbation function $v \in w_U + H^1_0(B_r)$.

Let us also make note of the one additional almost minimality property.     For any $z \in H^1_0(B_r(x_0))$
  \[ |\mathcal{J}_{g/\mu}(z,B_r(x_0)) - \mathcal{J}_{g(x_0)/\mu}(z,B_r(x_0))| \leq (\osc_{B_r(x_0)} g/\mu)|B_r| \leq \omega_g(\mu^{\frac{1}{d+2}}r)|B_r|.\]
  By \lref{scaling-bounds} $\mu^{\frac{1}{d+2}} \leq C(\Lambda,d)|U|^{\frac{1}{d}}$.  So up to an alteration of the definition of $R_0$ by a universal constant we have, for $v \in w_U + H^1_0(B_r)$ the $a$-harmonic replacement,
\[ \mathcal{J}_{g(x_0)/\mu}(w_U,B_r) \leq \big(1+r/R_0\big)\mathcal{J}_{g(x_0)/\mu}(v,B_r)+\omega_g(r/R_0)|B_r|. \] 
An analogous result for arbitrary test perturbations holds when $w_U$ is Lipschitz by the same argument.

\medskip

{\bf Step 3.} (Lipschitz estimate)  Next we obtain a Lipschitz estimate of the eigenfunction.

\begin{corollary}\label{c.uU-Lip}
The principal eigenfunction $u_U$ is Lipschitz with (scale invariant) bound
\[ |U|^{\frac{1}{2}+\frac{1}{d}}\|\grad u_U\|_{L^\infty} \leq C(\Lambda,d,\gamma,\omega_g,\|\grad a\|_\infty).\]
\end{corollary}
\begin{proof}
Fix an $x_0 \in U$.  \lref{almostminprop} says that $w_U = \mu^{-1/2}u_U$ has an almost minimality property of the type used in \tref{almostmin-lip-est} (i.e. perturbations by $a$-harmonic replacement) with $Q(x) = g(x)/\mu$, which has $(1+\gamma)^{-1} \leq Q \leq 1+\gamma$, and $R_0 = c|U|^{\frac{1}{d}}$.  Therefore, applying \tref{almostmin-lip-est},
\[ |\grad w_U(x_0)| \leq C(d,\Lambda,\gamma,\omega_g,\|\grad a\|_\infty)(1+\|\grad w_U\|_{\underline{L}^2(B_{R_0})}).\]
Using the scaling relations in \lref{scaling-bounds} gives 
\[ \|\grad w_U\|_{\underline{L}^2(B_{R_0})} \leq C|U|^{-1/2}\mu^{-1/2}\|\grad u_U\|_{L^2(\R^d)} \leq C.\]
Note how important it is that we establish the almost minimality condition up to the maximum length scale $\sim |U|^{\frac{1}{d}}$ and that the Lipschitz iteration starts from this large scale not a small scale $1$.
\end{proof}

{\bf Step 4.} (Non-degeneracy, perimeter bound, and density estimates) Next we establish some initial regularity of the free boundary in the form of non-degeneracy, a Hausdorff dimension $(d-1)$ bound, and inner and outer density estimates.  This step requires the Lipschitz estimate but is otherwise quite standard. 

\begin{corollary}\label{c.weak-domain-regularity}
There are constants $\ell$, $\kappa_0$, $P$, and $I$ depending on $(\Lambda,d,\|\grad a\|_\infty,\gamma,\omega_g)$ so that 
\begin{enumerate}[label = (\roman*)]
\item The $a$-eigenfunction $u_U$ has the scale invariant $\ell$-non-degeneracy property \dref{ell-non-degen}.
\item The domain $U$ has inner/outer density bound $\kappa_0$, as in \dref{scale-inv-densities}.
\item The domain $U$ has scale invariant boundary strip area bound with constant $P$ as in \dref{perimeter-hyp}.
\item The domain $U$ can be covered by $I$ balls of radius $|U|^{\frac{1}{d}}$.
\end{enumerate}
In particular $U$ is weakly regular as in \dref{weakly-regular}.
\end{corollary}

Most of the proof is covered by applying \pref{AM-misc-reg} using the almost minimality of $w_U$ in \lref{almostminprop}. However we do need to argue a bit for (iii) and (iv) so we postpone the proof to \sref{weak-domain-regularity}.

{\bf Intermezzo on existence on $\R^d$.}  At this stage we have only known about existence of minimizers of $J_g$ when we are set on large tori $\R^d /N\Z^d$.  Now, given the regularity theory developed, we can give a proof of existence of minimizers on the whole space $\R^d$.

\begin{theorem}\label{t.global-existence}
Suppose that $g$ satisfies the hypotheses \eref{g-hyp1}, \eref{g-hyp2} and \eref{g-hyp3}.  Then there exists $U$ open minimizing $J_g$ over bounded quasi-open subsets of $\R^d$.
\end{theorem}

The proof, which can be found below in \sref{global-existence}, is essentially using \cref{weak-domain-regularity} part (iv) to divide $\R^d/N\Z^d$ minimizers $U_N$ into finitely many components and then applying a period translation to each to put all the components inside a fixed compact region.  This is also where the localizing property of $g$ \eref{g-hyp3} is used.

\medskip

{\bf Step 5.} (Suboptimal rate of homogenization for initial flatness)  We cannot go directly from the Lipschitz estimate of $u$ to free boundary regularity.  The central issue, which is often present in free boundary and interface regularity problems, is the possibility of singular points even for homogeneous energies \cite{JerisonSavin}.  We can rule this out by an explicit identification of the homogenization limit with rate of convergence.  If the homogenized limit has all regular free boundary points then the minimizer $U$ at large scales, small $\mu$, will be flat in a neighborhood of each boundary point at the largest length scale $|U|^{1/d}$.

Notice that we have an error coming from homogenization and from $g$ non-constant.  There is not much advantage at this stage to separate out the two estimates. The problem with non-constant $g$ but homogeneous $\bar{a}$ is also inheriting its regularity from the closeness to $J_\mu$.

\begin{proposition}\label{p.generic-Jg-errorest}
  Let $U$ be a $J_g$ minimizer.  Let $E = \bar{a}^{1/2}B$ be the $\lambda_1(\cdot,\bar{a})$ minimizing ellipsoid of the same volume as $U$ and infimizing $|U \Delta E|$ the class of $\bar{a}$-ellipsoids with $|U| = |E|$. Call
\begin{equation}\label{e.eigen-special-error} \mathcal{E}(U) =  |{E}|^{2/d}(\lambda_1({E},a) - \lambda_1({E},\bar{a}))+|U|^{2/d}(\lambda_1(U,\bar{a}) - \lambda_1(U,a))
\end{equation}
to be a scaled error term. Then there is $C (\Lambda,d) \geq 1$ so that
\begin{enumerate}[label = (\alph*)]
\item (Eigenvalue estimates)
\[|U|^{\frac{2}{d}}|\lambda_1(U,a) - \lambda_1({E},\bar{a})| \leq C(\mathcal{E} + \gamma)\]
and
 \[ |U|^{\frac{2}{d}}(\lambda_1(U,\bar{a}) - \lambda_1({{E}},\bar{a})) \leq C(\mathcal{E} + \gamma).\]
 \end{enumerate}
 For the remaining parts we assume $(\mathcal{E}+\gamma) \leq 1$, and $E$ can be taken \emph{either} as above \emph{or} to be the ellipsoid minimizing $J_\mu(\cdot,\bar{a})$ and minimizing $|E\Delta U|$ over that class.
 \begin{enumerate}[resume,label=(\alph*)]
 \item (Measure estimates)
 \[ \frac{|U \Delta E|}{|U|} \leq C(\mathcal{E} + \gamma)^{\frac{1}{2}} \]
\item (Domain distance estimate)
\[ |U|^{-\frac{1}{d}}d_H(\partial U, \partial {E}) \leq C(\mathcal{E} + \gamma)^{\frac{1}{2d}}  \]
\item (Eigenfunction distance estimate)
\[ |U|^{1/2}\|u_{U}(\cdot,a) - u_{{E}}(\cdot,\bar{a})\|_\infty \leq C(\mathcal{E}+\gamma)^{\frac{1}{2d(d+2)}}+C|U|^{-\frac{1}{2d(d+2)}}\]
\end{enumerate}
\end{proposition}

The proof can be found below in \sref{generic-Jg-errorest}.  We give a brief sketch here.  The eigenvalue estimates in part (a) follow from a direct energy argument.  Then part (a) with Faber-Krahn stability \tref{FaberKrahn} implies proximity to an $\bar{a}$-ellipsoid in measure.  This can then be upgraded to an estimate in Hausdorff distance via \lref{Linftyupgrade} using that $U$ is weakly regular, in the sense of \dref{weakly-regular}.  From a homogenization argument, found in \pref{eigenfunction-conv}, we can then upgrade further to proximity of the eigenfunctions.

Now the issue is to control the eigenvalue error term $\mathcal{E}(U)$.  This is where the $L^2$-theory of periodic homogenization discussed in \sref{L2theory} comes into play.  The domain regularity of $U$ plays an essential role.  In order to get even a sub-optimal quantitative estimate on the eigenvalue error $|U|^{2/d}(\lambda_1(U,\bar{a}) - \lambda_1(U,a))$ one already needs a certain amount of domain regularity.  The notion of weakly regular \dref{weakly-regular}, i.e. domain inner and outer density estimates and boundary strip area bound, is enough.

\begin{corollary}\label{c.suboptimal-est-eigenfunc}
  Let $U$ be a $J_g$ minimizer. There are $C \geq 1$ and $\alpha \in (0,1)$ depending on $(\Lambda,d,\|\grad a\|_\infty\gamma,\omega_g)$ so that
  \[ |U|^{-\frac{1}{d}}d_H(\partial U,\partial E) \leq C\gamma^{\frac{1}{2d}} + C|U|^{-\alpha}\]
  and
\[ |U|^{1/2}\|u_{U}(\cdot,a) - u_{E}(\cdot,\bar{a})\|_\infty \leq C\gamma^{\frac{1}{2d(d+2)}}+C|U|^{-\alpha}\]
where $E$ is the $\bar{a}$-ellipsoid minimizing $|E \Delta U|$ over the class of $J_\mu(\cdot,\bar{a})$ minimizers.
\end{corollary}
\begin{proof}
Since, by \cref{weak-domain-regularity}, $U$ is weakly regular with constants depending on $(\Lambda,d,\|\grad a\|_\infty,\gamma,\omega_g)$ we can apply the suboptimal quantitative homogenization estimates \cref{suboptimal-eigen-conv} (to $U$) and \cref{optimal-eigen-conv} (to $E$) to find
\[ \mathcal{E}(U) \leq C|U|^{-\alpha}\]
where $\mathcal{E}(U)$ was the eigenvalue error defined above in \eref{eigen-special-error}, and $C \geq 1$ and $\alpha \in (0,1)$ depend on $(\Lambda,d,\|\grad a\|_\infty,\gamma,\omega_g)$. Then apply \pref{generic-Jg-errorest}.
\end{proof}

Now the key information from \cref{suboptimal-est-eigenfunc} is that $L^\infty$ proximity to $u_E(\cdot,\bar{a})$ implies large scale flatness (because $u_E(\cdot,\bar{a})$ is smooth and has the correct free boundary condition).  
\begin{lemma}[Initial flatness]\label{l.flatness}
  For any $\delta>0$ there are $\gamma_0>0$ and $m_0>0$ depending on $(\Lambda,d,\|\grad a\|_\infty,\omega_g,\delta)$ and $r_0>0$ depending on $(\Lambda,d,\delta)$ so that if $\gamma \leq \gamma_0$ and $U$ is any $J_g$ minimizer with $|U| \geq m_0$ then for any $x_0 \in \partial U$
\[ \inf_{\nu \in S^{d-1}}\sup_{x \in B_{R}(x_0)} \frac{1}{R}|w_U(x) - \frac{1}{(\nu \cdot \bar{a} \nu)^{1/2}} (x \cdot \nu)_+| \leq \delta  \ \hbox{ for some choice of }  R \geq r_0|U|^{\frac{1}{d}}.\]
Here $w_U = \mu^{-1/2}u_U$. 
\end{lemma}
Note that the lower bound condition $|U| \geq m_0$ could equivalently be stated as an upper bound on $\mu \leq \mu_0(m_0)$ given the scaling relations in \lref{scaling-bounds}.

The proof is postponed to \sref{flatness}.  For a sketch: pick $r_0$ small enough so that the flatness condition holds for $u_E(\cdot,\bar{a})$ at scales $R \leq r_0|E|^{1/d}$.  Then use triangle inequality with the $L^\infty$ convergence in \cref{suboptimal-est-eigenfunc}.

\medskip

{\bf Step 6.} (Flatness implies large scale regular free boundary) At this penultimate stage we can combine the almost minimality property \lref{almostminprop} with the initial flatness provided by \lref{flatness} to iterate and get a (large scale) Lipschitz domain property for $U$.

Apply \tref{REGmain2} in combination with the initial flatness given by \lref{flatness} and the almost minimality property \lref{almostminprop}, to find that $U$ is an $(r_0,1,h)$ (large scale) Lipschitz domain in the sense of \dref{large-scale-lipschitz}.  The constants $(r_0,h)$ depend on $(\Lambda,d,\|\grad a\|_\infty,\omega_g)$.

\medskip

{\bf Step 7.} (Optimal rate of homogenization)  Finally, with the Lipschitz domain regularity in hand, we can upgrade the rate of homogenization.

At this point we have shown that $U$ has the following properties:
\begin{enumerate}[label = (\roman*)]
 \item $a$-eigenfunction scale invariant $L$-Lipschitz estimate \dref{L-Lipschitz}.
  \item $a$-eigenfunction scale invariant $\ell$-non-degeneracy estimate \dref{ell-non-degen}.
 \item Domain $(r_0,1,h)$-(large scale) Lipschitz estimate \dref{large-scale-lipschitz}.
 \end{enumerate}
The parameters $(L,\ell,r_0,h)$ in the above property depend only on the input parameters $(\Lambda,d,\|\grad a\|_\infty,\omega_g)$.  

So the optimal (up to logarithms) quantitative homogenization result \cref{quanthom-large-scale-lip} gives us the estimate
\[ |U|^{2/d}(\lambda_1(U,\bar{a}) - \lambda_1(U,a))_+ \leq C|U|^{-\frac{1}{d}}|\log |U||^{\frac{1}{2}+\eta}\]
with constant $C\geq 1$ depending on $(\Lambda,d,\|\grad a\|_\infty,\omega_g,\eta)$. Plugging this back into \pref{generic-Jg-errorest} concludes the proof of \tref{augmented-full}. \qed

\subsection{Proof of \lref{almostminprop}}\label{s.almost-min-proof}

Let $u \in H^1(U)$ with $\|u\|_{L^2(U)} = 1$ and
\[  \int_{U} \grad u\cdot a(x) \grad u \ dx = \lambda_1(U,a).\]
 Let $v \in H^1(B_r)$ with $v \in u+ H^1_0(B_r)$ and extend $v$ to be equal to $u$ outside of $B_r$.  Then define
\[  Z = \min\{1,(\|u\|^2_{L^2(\R^d \setminus B_r)} + \|v\|^2_{L^2(B_r)})^{1/2}\} \ \hbox{ so that } \ \|Z^{-1}v\|_{L^2} \geq 1.\]
If $v \geq u$ is an upward perturbation then $Z=1$ and the arguments below vastly simplify.  Note that
\[ Z^{-2} - 1 =\frac{\|u\|^2_{L^2(B_r)}-\|v\|^2_{L^2(B_r)}}{1-[\|u\|^2_{L^2(B_r)}-\|v\|^2_{L^2(B_r)}]} \ \hbox{ when } \ \|v\|_{L^2(B_r)} \leq \|u\|_{L^2(B_r)}.\]
By Sobolev embedding,
 \begin{align*}
 \|u\|_{{L}^2(B_r)} &\leq \|u^2\|^{1/2}_{L^{\frac{d}{d-2}}(B_r)}\|1\|_{L^{\frac{d}{2}}(B_r)}^{1/2}\\
 &\leq \|u\|_{L^{\frac{2d}{d-2}}(\R^n)} |B_r|^{\frac{1}{d}} \\
 &\leq \|\grad u\|_{L^2(\R^n)}r\\
 &\leq C\lambda_1(U,a)^{1/2}r.
 \end{align*}
 So for $C\lambda_1(U,a)^{1/2}r \leq \frac{1}{2}$, i.e. by the scalings in \lref{scaling-bounds} $r \leq c(\Lambda,d)|U|^{\frac{1}{d}}$, we can bound
 \[ Z^{-2} - 1  \leq 2\int_{B_r} (u^2 - v^2) \ dx.\]
 Using Poincar\'e inequality since $u-v \in H^1_0(B_r)$
\begin{align*}
 \int_{B_r} (u^2 - v^2) \ dx &= \int_{B_r} (u-v)(u+v) \ dx \\
 &\leq \|u-v\|_{{L}^2(B_r)}(\|u\|_{{L}^2(B_r)} + \|v\|_{{L}^2(B_r)}) \\
 &\leq Cr\|\grad (u-v)\|_{{L}^2(B_r)}\|u\|_{{L}^2(B_r)} \\
 &\leq Cr\|\grad (u-v)\|_{{L}^2(B_r)}\|u\|_{L^\infty}|B_r|^{1/2}\\
 &\leq Cr\|u\|_{L^\infty}\|\grad (u-v)\|_{\underline{L}^2(B_r)}|B_r|
 \end{align*}
 we used $\|v\|_{L^2(B_r)} \leq \|u\|_{L^2(B_r)}$ whenever $Z \neq 1$ for the second inequality.  
 
 Then
\begin{align}
 0 &\leq J_g(\{Z^{-1}v>0\}) - J_g(\{u>0\}) \notag\\
 &=\mathcal{J}_g(v,B_r)- \mathcal{J}_g(u,B_r)+ (Z^{-2}-1)\int_{\R^d}  \grad v \cdot a(x) \grad v \ dx \notag\\
 &\leq \mathcal{J}_g(v,B_r)- \mathcal{J}_g(u,B_r)+Cr\|u\|_{L^\infty}\|\grad (u-v)\|_{\underline{L}^2(B_r)}\int_{\R^d}  \grad v \cdot a(x) \grad v \ dx|B_r|. \label{e.initial-almost-min-chain}
  \end{align}
We focus for a moment on the tricky term $\int_{\R^d}  \grad v \cdot a(x) \grad v \ dx$. Now if $v \in u + H^1_0(B_r)$ is the $a$-harmonic replacement of $u$ in $B_r$ then
\[ \int_{\R^d} \grad v \cdot a(x) \grad v \ dx \leq \int_{\R^d}  \grad u \cdot a(x) \grad u \ dx = \lambda_1(U,a)\]
while, in general, we can write
\[ \int_{\R^d} \grad v \cdot a(x) \grad v \ dx \leq \lambda_1(U,a)(1+\tfrac{C\|\grad v\|_{L^2(B_r)}^2}{\lambda_1(U,a)}).\]
We will continue writing the term $\frac{\|\grad v\|^2_{L^2(B_r)}}{\lambda_1(U,a)}$ but just recall that it is not necessary in the $a$-harmonic replacement case.  

Now we proceed with rearranging in \eref{initial-almost-min-chain} to find
\[ \mathcal{J}_g(u,B_r) \leq \mathcal{J}_g(v,B_r)+Cr\|u\|_{L^\infty}\lambda_1(U,a)(1+\tfrac{\|\grad v\|_{L^2(B_r)}^2}{\lambda_1(U,a)})\|\grad (u-v)\|_{\underline{L}^2(B_r)}|B_r|. \]

Next we make the rescaling argument to get a $\mathcal{J}_{g/\mu}$ almost minimal property.  Notice that
\[ \mathcal{J}_g(w,\Omega) = \mu \mathcal{J}_{g/\mu}(\mu^{-1/2}w,\Omega).\]
So, if $w = \mu^{-1/2}u$ and $v \in w + H^1_0(B_r)$ 
\begin{align*}
 \mathcal{J}_{g/\mu}(w,B_r) &= \mu^{-1} \mathcal{J}_{g}(u,B_r) \\
 &\leq  \mu^{-1} \mathcal{J}_{g}(\mu^{1/2}v,B_r)+Cr\|u\|_{L^\infty}\lambda_1(U,a)(1+\tfrac{\mu\|\grad v\|^2_{L^2(B_r)}}{\lambda_1(U,a)})\mu^{-1}\|\grad (u-\mu^{1/2}v)\|_{\underline{L}^2(B_r)}|B_r|\\
 &\leq \mathcal{J}_{g/\mu}(v,B_r)+Cr\|u\|_{L^\infty}\lambda_1(U,a)\mu^{-1/2}(1+\tfrac{\mu \|\grad v\|^2_{L^2(B_r)}}{\lambda_1(U,a)})\|\grad (w-v)\|_{\underline{L}^2(B_r)}|B_r|.
 \end{align*}
Now define $R_1 = C\|u\|_{L^\infty}\lambda_1(U,a)\mu^{-1/2}$ and notice that, by the scalings for $J_g$ minimizers found in \lref{scaling-bounds},
 \[ R_1^{-1} =C \mu^{-\frac{1}{2}} \|u\|_{L^{\infty}(B_r)}\lambda_1(U,a) \leq C|U|^{\frac{d+2}{2d}}|U|^{-\frac{1}{2}}|U|^{-\frac{2}{d}} = C|U|^{-\frac{1}{d}}\]
 and also $\mu\lambda(U,a)^{-1} \sim \mu^{\frac{d}{d+2}} \sim |U|^{-1}$, so in general we have
 \begin{equation}\label{e.form2}
 \mathcal{J}_{g/\mu}(w,B_r) \leq \mathcal{J}_{g/\mu}(v,B_r)+(r/R_1)(1+C\tfrac{\|\grad v\|^2_{L^2(B_r)}}{|U|})\|\grad (w-v)\|_{\underline{L}^2(B_r)}|B_r| 
 \end{equation}
 and in the case when $v$ is the $a$-harmonic replacement the form is a bit simpler
  \begin{equation}\label{e.form1}
    \mathcal{J}_{g/\mu}(w,B_r) \leq \mathcal{J}_{g/\mu}(v,B_r)+(r/R_1)\|\grad (w-v)\|_{\underline{L}^2(B_r)}|B_r| 
 \end{equation}

 This is already a good notion of almost minimality but we make some additional manipulations just to get exactly to the form used \cite{FeldmanReg}.  Note that this form is better in the sense that it respects the scaling argument we just made, while  the form in \cite{FeldmanReg} would not have.  In any case, we just use 
 \begin{align*}
  &\|\grad (w-v)\|_{\underline{L}^2(B_r)}|B_r| \\
  &\hspace{7ex}\leq [\|\grad w\|_{\underline{L}^2(B_r)}+\|\grad v\|_{\underline{L}^2(B_r)}]|B_r| \\
  &\hspace{7ex}\leq [2 + \|\grad w\|_{\underline{L}^2(B_r)}^2+\|\grad v\|_{\underline{L}^2(B_r)}^2]|B_r| \\
  &\hspace{7ex}\leq 2|B_r| + C\mathcal{J}_{g/\mu}(w,B_r)+C\mathcal{J}_{g/\mu}(v,B_r)
  \end{align*}
  so now we conclude the argument in the $a$-harmonic replacement case plugging into \eref{form1}
  \[ (1-Cr/R_1)\mathcal{J}_{g/\mu}(w,B_r) \leq (1+Cr/R_1)\mathcal{J}_{g/\mu}(v,B_r) + 2(r/R_1)|B_r|\]
  and for $r \leq cR_1$ we can divide through by $(1-Cr/R_1)$.

 {\bf Case that $w$ is Lipschitz.} In the second case, when $v$ is not the $a$-harmonic replacement but instead $w$ is known to be $L$-Lipschitz we proceed from the inequality \eref{form2} bounding $\|\grad v\|_{L^2(B_r)}^2 \leq C\mathcal{J}_{g/\mu}(v,B_r)$:
  \begin{align*}
    &\left[1-(r/R_2)(1+ |U|^{-1}\mathcal{J}_{g/\mu}(v,B_r))\right]\mathcal{J}_{g/\mu}(w,B_r) \leq \\
    &\hspace{10ex}\left[1+(r/R_2)(1+ |U|^{-1}\mathcal{J}_{g/\mu}(v,B_r))\right]\mathcal{J}_{g/\mu}(v,B_r)+(r/R_2)(1+ |U|^{-1}\mathcal{J}_{g/\mu}(v,B_r))|B_r| 
    \end{align*}
    for an appropriately defined $R_2 = c|U|^{1/d}$. Notice that 
    \[\mathcal{J}_{g/\mu}(v,B_r) = \frac{|B_r|}{|U|} |B_r|^{-1}\mathcal{J}_{g/\mu}(v,B_r)\] and $|B_r|/|U| \leq (r/R_2)^d$ if we decrease the constant in the definition of $R_2$ if necessary.  
    
    Now we make a purely calculus argument to simplify the almost minimizer inequality. If we call 
    \[X = |B_r|^{-1}\mathcal{J}_{g/\mu}(w,B_r) \leq C(1+L^2), \ Y = |B_r|^{-1}\mathcal{J}_{g/\mu}(v,B_r), \ \hbox{ and } \ \delta = (r/R_2)\]
     then we can write the previous inequality more compactly as
    \[ [1-\delta(1+\delta^dY)]X \leq [1+\delta(1+\delta^dY)]Y+\delta(1+\delta^dY).\]
    Our aim is to reduce this to an inequality of the form
    \[ X \leq Y + C\delta.\]
    First we argue that we can divide through by $[1-\delta(1+\delta^dY)]$.  If $\delta(1+\delta^dY) \geq \frac{1}{2}$ then $X \leq C(1+L^2) \leq 2C(1+L^2)[\delta^{d+1}Y + \delta]$ which is stronger than the conclusion we desire as long as $C(1+L^2) \delta^{d+1} \leq 1$ which we can guarantee by decreasing $R_2 = c(L,\Lambda,d)|U|^{1/d}$.
    
    Now we arrive at the inequality
    \[ X \leq f(Y) \ \hbox{ with } \ f(Y) = \frac{1+\delta(1+\delta^dY)}{1-\delta(1+\delta^dY)}Y+\delta(1+\delta^dY)\]
    where $f$ is a monotone increasing function of $Y$ on $\R_+$.  Then this means that
    \[ f^{-1}(X) \leq Y.\]
 Now if $X \leq \delta$ we have the desired conclusion, otherwise, by fundamental theorem of calculus, we can write
    \[ f^{-1}(\delta) + \int_\delta^X (f^{-1})'(S) \ dS \leq Y,\]
    and we have the convenience $f^{-1}(\delta) = 0$.  So we want a lower bound on $(f^{-1})'$ on $[\delta,C(1+L^2)] \supset [\delta,X]$.   This amounts to an upper bound of $f'$ on $[0,f(C(1+L^2))] \subset [0,\tilde{C}(1+L^2)]$ which is
    \[ f'(Y)  \leq 1 + C(L)\delta \ \hbox{ on } Y \in [\delta,f(C(1+L^2))]\]
    so
    \[ (f^{-1})'(S) \geq 1-C(L)\delta \ \hbox{ on } \ S \in [\delta,C(1+L^2)].\]
    Thus
    \[ (1-C(L)\delta)X \leq Y\]
    and using $X \leq CL$ again
    \[ X \leq Y + C(L)\delta.\]

  \qed

\subsection{Proof of \cref{weak-domain-regularity}}\label{s.weak-domain-regularity}

Apply \pref{AM-misc-reg} to $w_U = \mu^{-\frac{1}{2}}u_U$ which has, via \cref{uU-Lip}, $\|\grad w_U\|_\infty \leq C(\Lambda,d,\gamma,\omega_g,\|\grad a\|_\infty)$.

For the perimeter bound (iii) it seems that the simplest thing is to argue directly, the general almost minimizer argument \lref{almostminprop} seems to give up a bit too much information.  Consider the perturbation
\[ v(x) = (u_U(x) - t\|\grad u\|_\infty)_+\]
which has
\[ \|u\|_{L^2}^2-\|v\|_{L^2}^2  = \int (u-v)(u+v) \ dx \leq \|u-v\|_{L^2}\|u+v\|_{L^2} \leq 2t \|\grad u\|_\infty |U|^{1/2} \]
and
\[ \|v\|_{L^2}^{-2} \leq 1 + Ct \|\grad u\|_\infty|U|^{1/2}\]
as long as 
\[t\|\grad u\|_\infty|U|^{1/2} \leq \frac{1}{2}.\]
  Given \cref{uU-Lip} this requirement is $t \leq c(\Lambda,d)|U|^{\frac{1}{d}}$.   So doing energy comparison with $V = \{v>0\}$
\[ \int_{\{0 < u_U < t\|\grad u_U\|_\infty\}} a(x) \grad u_U \cdot \grad u_U+ g(x) \ dx \leq Ct \|\grad u_U\|_\infty|U|^{1/2}\lambda_1(U,a).\]
Since the left hand side dominates $(1+\gamma)^{-1}\mu |\{0 < u_U < t\|\grad u_U\|_\infty\}|$ we can find
\[  |\{0 < u_U < t\|\grad u_U\|_\infty\}| \leq Ct \|\grad u_U\|_\infty|U|^{1/2}\mu^{-1}\lambda_1(U,a) \leq Ct|U|^{-\frac{1}{d} - \frac{2}{d} + \frac{d+2}{d}} = Ct|U|^{\frac{d-1}{d}}\]
Lipschitz estimate \cref{uU-Lip} then implies
\[|\{x \in U : d(x,\partial U) < t\}| \leq |\{0 < u_U < t\|\grad u_U\|_\infty\}| \leq Ct|U|^{\frac{d-1}{d}}\]

The only thing left is the covering bound (iv). Notice that the inner density estimates with $r = |U|^{\frac{1}{d}}$ imply
\[ |U \cap B_r(z)|\ \geq \kappa_0 |B_r| \ \hbox{ for all } \ z \in U\]
with $\kappa_0(\Lambda,d,\|\grad a \|_\infty,\gamma,\omega)>0$.  So if $\{B_r(z_i)\}_{i=1}^I$ is an $M$ overlapping covering of $U$ then
\[  \kappa_0 I |B_r| \leq \sum_{i=1}^I |U \cap B_r(z)| \leq M |U|\]
meaning $I \leq M\kappa_0^{-1}$.  So by Besicovitch covering we get the claim.
\qed

\subsection{Proof of \tref{global-existence}}\label{s.global-existence}
\begin{proof}
Extend $g$ to be $N\Z^d$-periodic in the way explained in \rref{g-hyp4}.  Let $U_N$ be a minimizer of $J_g$ over quasi-open subsets of $\R^d / N \Z^d$ for $N \geq R_g \vee \mu^{-\frac{1}{d+2}}$ integer.  

By \cref{weak-domain-regularity} $U_N$ can be covered by a constant $I$ number of balls of radius $|U_N|^{\frac{1}{d}} \sim \mu^{-\frac{1}{d+2}}$.  Call $V_N$ to be the union of these $I$ balls, this is an open set with $I' \leq I$ many connected components $V_{N,1},\dots,V_{N,I'}$ each of diameter at most $C \mu^{-\frac{1}{d+2}}$, and call $U_{N,i} = U_N \cap V_{N,i}$.  

Thus there is an $N_0 \geq R_g+C\mu^{-\frac{1}{d+2}}$ sufficiently large so that we can apply a separate $\Z^d$-lattice translation to each $U_{N,i}$ (canonically apply no translation if $V_{N,i}$ was already contained in $[-N_0/2,N_0/2)^d$) so that the translated $U_{N,i}$ are disjoint and all contained in $[-N_0/2,N_0/2)^d$ and so that $g \equiv \sup g$ on any $U_{N,i}$ which intersects the complement of $[-N_0/2,N_0/2)^d$.  Call the new set created by the disjoint union of the $U_{N,i}$ to be $\tilde{U}_N$.

By the $\Z^d$ periodicity of $a$ this translation does not affect the $\lambda_1(U_{N,i},a)$ part of the energy.  The $\int_{U_{N,i}} g \ dx$ term in the energy cannot be increased: if $U_{N,i}$ intersected the complement of $[-N_0/2,N_0/2)^d$ then $g  \equiv \sup g$ there and translation could only decrease that term in the energy.  Thus $\tilde{U}_N$ is also a $J_g$ minimizer.

Thus we now have a sequence of $J_g$ minimizers among quasi-open subsets of $\R^d / N\Z^d$, $\tilde{U}_N$, which are all contained in a fixed compact region $[-N_0/2,N_0/2]^d$.  A subsequential limit $U$ which minimizes $J_g$ over bounded quasi-open subsets of $\R^d$ can be extracted by a typical compactness / lower-semi-continuity argument (see \cite{VelichkovBook}[page 4-5]).
\end{proof}

\subsection{Proof of \pref{generic-Jg-errorest}}\label{s.generic-Jg-errorest}

Let ${E}$ be the $\lambda_1(\cdot,\bar{a})$ minimizing ellipsoid with volume $|U|$. Energy comparison gives

  \begin{align*}
  \lambda_1({E},\bar{a}) & \leq \lambda_1(U,\bar{a})  \\
  & = \lambda_1(U,a) + (\lambda_1(U,\bar{a}) - \lambda_1(U,a))\\
  &= J_g(U,a) - \int_U g(x) \ dx + (\lambda_1(U,\bar{a}) - \lambda_1(U,a))\\
  &\leq J_g({E},a) - \int_U g(x) \ dx + (\lambda_1(U,\bar{a}) - \lambda_1(U,a)) \\
  & = \lambda_1({E},a) + \int_{E} g(x) \ dx - \int_U g(x) \ dx + (\lambda_1(U,\bar{a}) - \lambda_1(U,a))\\
  &\leq \lambda_1({E},\bar{a}) +  2[(1+\gamma) - (1+\gamma)^{-1}] \mu |U|+(\lambda_1({E},a) - \lambda_1({E},\bar{a}))+(\lambda_1(U,\bar{a}) - \lambda_1(U,a)).
  \end{align*}
  Using for the last line that $|\int_U g(x) \ dx - \mu |U|| \leq \mu\|\frac{g}{\mu}-1\|_\infty|U|$ and the same for $\int_{E} g(x) \ dx$ which has $|{E}| = |U|$. 
  
  Contained in this sequence of inequalities we derived
  \[ \lambda_1({E},\bar{a}) -\lambda_1(U,a)  \leq (\lambda_1(U,\bar{a}) - \lambda_1(U,a))\]
  and
  \[ \lambda_1(U,a)- \lambda_1({E},\bar{a}) \leq C\gamma \mu |U| + (\lambda_1({E},a) - \lambda_1({E},\bar{a}))\]
  and finally
\[ 0 \leq \lambda_1(U,\bar{a}) - \lambda_1({E},\bar{a}) \leq C\gamma \mu |U|+(\lambda_1({E},a) - \lambda_1({E},\bar{a}))+(\lambda_1(U,\bar{a}) - \lambda_1(U,a))\]
which, given the scaling properties in \lref{scaling-bounds}, is the claimed estimate.

We can also conclude, using previous inequalities again, that
  \begin{align*}
   J_\mu(E,\bar{a}) &= J_\mu(U,a) + \lambda_1({E},\bar{a}) -\lambda_1(U,a) \\
   & \leq J_g(U,a) + C(\mathcal{E} + \gamma)|U|^{-\frac{2}{d}} \\
   &\leq J_g(E_\mu,a) + C(\mathcal{E} + \gamma)|U|^{-\frac{2}{d}} \\
   &\leq J_\mu(E_\mu,a) + C(\mathcal{E} + \gamma)|U|^{-\frac{2}{d}} \\
   & \leq J_\mu(E_\mu,\bar{a}) + (\lambda_1({E}_\mu,a) - \lambda_1({E}_\mu,\bar{a}))+ C(\mathcal{E} + \gamma)|U|^{-\frac{2}{d}}\\
   &\leq J_\mu(E_\mu,\bar{a}) + C(\mathcal{E} + \gamma)|U|^{-\frac{2}{d}}.
   \end{align*}
   where $E_\mu$ is the $J_\mu(\cdot,\bar{a})$ minimizing $\bar{a}$-ellipsoid.  For the last inequality we are using that, since $E_\mu$ is a dilation of $E$ by (already) factor bounded by a universal constant, this middle error term can also be bounded by $C\mathcal{E}$.  By the computations in \sref{Jmu-minimizers-computations} $E_\mu = \bar{a}^{1/2}B_{\rho_*(\mu)}$ with
   \[ \rho_*^{d+2} = \tfrac{2}{d}\mu^{-1}\textup{det}(\bar{a})^{-\frac{1}{2}}|B_1|^{-1}\lambda_1(B_1,\textup{id})\]
   and $J_\mu(\bar{a}^{1/2}B_\rho,\bar{a})$ is a strictly convex function of $\rho$ achieving its minimal value at $\rho = \rho_*(\mu)$ so
   \[  \frac{||U| - |E_\mu||}{|U|} \leq C(\mathcal{E} + \gamma)^{1/2}\ \hbox{ and } \ \frac{|\rho - \rho_*(\mu)|}{|U|^{\frac{1}{d}}} \leq C(\mathcal{E} + \gamma)^{1/2}.\]
   This justifies that we can take $E$ for the remainder of the proof to either be the minimizer of $|E \Delta U|$ over the class $|E| = |U|$ \emph{or} over the class of $J_\mu(\cdot,\bar{a})$ minimizers.
   
Next we apply Faber-Krahn stability \tref{FaberKrahn} and \rref{faber-krahn-remark} below it implies
  \[ c_d \left(\frac{|U \Delta {E}|}{|U|}\right)^{2} \leq |U|^{\frac{2}{d}}\lambda_1(U,\bar{a}) - |E|^{\frac{2}{d}}\lambda_1({E},\bar{a}) \leq \mathcal{E} + C\gamma.\]

Now we use the regularity theory that we have established, from \cref{weak-domain-regularity} $U$ has the (scaled) eigenfunction $L$-Lipschitz, eigenfunction $\ell$-non-degenerate, and inner and outer density estimates with universal $L$, $\ell$ and $\kappa_U$. Then \lref{Linftyupgrade} implies
\[\frac{d_H(\partial U,\partial {E})}{|U|^{\frac{1}{d}}} \leq C\left(\frac{|U \Delta {E}|}{|U|}\right)^{\frac{1}{d}} \leq C(\mathcal{E}+\gamma)^{\frac{1}{2d}}\]
Finally \pref{eigenfunction-conv} implies
\[ |U|^{\frac{1}{2}}\|u_U - u_{E}\|_{L^\infty} \leq C(\mathcal{E}+\gamma)^{\frac{1}{2d(d+2)}} + C|U|^{-\frac{1}{2d(d+2)}}.\]
\qed

\subsection{Proof of \lref{flatness}}\label{s.flatness} 
Let $E$ be the ellipsoid minimizing $J_\mu(\cdot,\bar{a})$ and minimizing $|U \Delta E|$.  Note that $E = \bar{a}^{1/2}B_{\rho}$ with
\[ \rho^{d+2} = \tfrac{2}{d}\mu^{-1}\textup{det}(\bar{a})^{-\frac{1}{2}}|B_1|^{-1}\lambda_1(B_1,\textup{id})\] 
and so
\[ u_{E}(x) = u_{B_\rho}(\bar{a}^{-1/2}x) = \rho^{-d/2}u_{B_1}(\bar{a}^{-1/2}x/\rho).\]
Now $u_{B_1}$ has
\[ \grad u_{B_1}(x) = -\beta_d x \ \hbox{ for } \ x \in \partial B_1\]
where, by \eref{ball-eigenfunction-slope}, $\beta_d = \sqrt{\frac{2}{d}|B_1|^{-1}\lambda_1(B_1,\textup{id})}$. Note that, of course, $\grad u_E(x)$ is parallel to the inward unit normal $\nu$ to $E$ at $x \in \partial E$ because $\partial E$ is the zero level set of $u_E$, thus
\[ \nu_x = \frac{\grad u_E}{|\grad u_E|} = \frac{\bar{a}^{-1}x}{|\bar{a}^{-1}x|}\]  
so 
\[ (\nu_x \cdot \bar{a} \nu_x)^{1/2} = |\bar{a}^{1/2}\nu_x| = |\bar{a}^{-1}x|^{-1}|\bar{a}^{-1/2}x| = |\bar{a}^{-1}x/\rho|^{-1} \]
since for $x \in \partial E$, $\bar{a}^{-1/2}x \in \partial B_\rho$.  And so
\begin{align*}
 |\grad u_E(x)| &= \rho^{-1-\frac{d}{2}}|\bar{a}^{-1/2} \grad u_{B_1}(\bar{a}^{-1/2}x/\rho)| \\
 &= \beta_d\rho^{-1-\frac{d}{2}}|\bar{a}^{-1}x/\rho|  \\
 &= \beta_d\rho^{-1-\frac{d}{2}}\frac{1}{(\nu_x \cdot \bar{a} \nu_x)^{1/2}}.
\end{align*}
If we change into the variable $w_E = \mu^{-1/2}u_E$ then, noting that $\mu^{-1/2} \beta_d\rho^{-1-\frac{d}{2}} = 1$,
\[ |\grad w_E(x)| = \frac{1}{(\nu_x \cdot \bar{a} \nu_x)^{1/2}} \ \hbox{ on } \ x \in \partial E.\]
Since $u_{B_1}$ is smooth up to $\partial B_1$, for every $\delta>0$ there is $r_0(\delta,d)>0$ so that for each $y_1 \in \partial B_1$
\[ \sup_{B_r(y_1)}r^{-1}|u_{B_1}(y) - \beta_d ((y-y_1) \cdot \nu^{B_1}_{y_1})_+| \leq \frac{\delta}{3} \ \hbox{ for all } \ r \leq r_1.\]
And rescaling this, for all $x_1 \in \partial E$
\begin{equation}\label{e.wE-flat}
 \sup_{B_R(x_1)}R^{-1}\left|w_{E}(x) -  \frac{1}{(\nu \cdot \bar{a} \nu)^{1/2}}((x-x_1) \cdot \nu)_+\right| \leq \frac{\delta}{3} \ \hbox{ for all } \ R \leq c(\Lambda,d)r_1|E|^{\frac{1}{d}}
 \end{equation}
with $\nu = \nu^E_{x_1}$ and call $r_2 = c(\Lambda,d)r_1$.

  Note the scaling relation
\begin{align*}
 |U|^{-\frac{1}{d}}\|w_U - w_E\|_\infty &= |U|^{-\frac{1}{d}}\mu^{-1/2}\|u_U - u_E\|_\infty \leq C|U|^{\frac{d+2}{2d}-\frac{1}{d}}\|u_U - u_E\|_\infty \\
 &= C|U|^{1/2}\|u_U - u_E\|_\infty
 \end{align*} 
 so that, applying \cref{suboptimal-est-eigenfunc},
\[ |U|^{-\frac{1}{d}} d_H(\partial U,\partial E) + |U|^{-\frac{1}{d}}\|w_U - w_E\|_\infty \leq C_0(\gamma^{\frac{1}{2d(d+2)}} + |U|^{-\alpha}).\]

If $x_0 \in \partial U$ there is a point $x_1 \in \partial E$ with
\[ |U|^{-\frac{1}{d}}|x_1 - x_0| \leq C_0(\gamma^{\frac{1}{2d}} + |U|^{-\alpha}).\]
Take $\nu = \nu^E_{x_1}$ and
\begin{equation}\label{e.R-condition}
  3\delta^{-1}C_0(\gamma^{\frac{1}{2d}} + |U|^{-\alpha}) |U|^{\frac{1}{d}} \leq R \leq \frac{1}{2}r_2|U|^{\frac{1}{d}}
  \end{equation}
if such $R$ exists. We have, since $R \geq 3\delta^{-1}|x_1 - x_0| \geq |x_1 - x_0|$,
\begin{align*}
&\sup_{B_R(x_0)}\frac{1}{R}\left|u_U(x) -  \frac{1}{(\nu \cdot \bar{a} \nu)^{1/2}}((x-x_0) \cdot \nu)_+\right| \\
& \hspace{10ex} \leq 2\sup_{B_{2R}(x_1)}\frac{1}{2R}\left|u_U(x) -  \frac{1}{(\nu \cdot \bar{a} \nu)^{1/2}}((x-x_0) \cdot \nu)_+\right|\\
& \hspace{10ex}  \leq 2\sup_{B_{2R}(x_1)}\frac{1}{2R}\left|w_E(x) -  \frac{1}{(\nu \cdot \bar{a} \nu)^{1/2}}((x-x_0) \cdot \nu)_+\right| + R^{-1}C_0(\gamma^{\frac{1}{2d(d+2)}} + |U|^{-\alpha})|U|^{\frac{1}{d}}\\
& \hspace{10ex}  \leq  \frac{2\delta}{3}+\frac{\delta}{3} = \delta.
\end{align*}
So we just need to choose $\gamma \leq \gamma_0$ and $|U| \geq m_0$ such that
\[ \gamma^{\frac{1}{2d}} + |U|^{-\alpha}  \leq C_0^{-1}\frac{1}{6}r_2 \delta\]
so that there exists an $R$ in the interval \eref{R-condition}.  Note these choices $0 < \gamma_0 \leq 1$ and $m_0 \geq 1$ depend on the parameters $(\Lambda,d,\|\grad a\|_\infty,\omega_g,\delta)$ because that is what $C_0(\Lambda,d,\|\grad a\|_\infty,\omega_g)$ and $r_2(\Lambda,d,\delta)$ depend on, but then the lower bound on $R$ in \eref{R-condition} is $r_0\frac{1}{2}r_2$ which only depends on $(\Lambda,d,\delta)$. 

 \section{Relation between augmented and volume constrained minimization} \label{s.relation}

We consider in this section the (non-trivial) relationship between 
the hard constraint problem
\begin{equation}\label{e.constrained}
\inf \{ \lambda_1(U,a) : \ \hbox{ $U$ is quasi-open and $|U| \leq m$}\}.
\end{equation}
and the soft constraint / augmented / Lagrange multipliers minimization problem
\begin{equation}\label{e.augmented}
\inf \{ J_\mu(U,a) : \ \hbox{ $U$ is quasi-open}\}.
\end{equation} 

All that is formally guaranteed by the method of Lagrange multipliers is that a volume constrained minimizer is an augmented functional \emph{critical point} for some value of $\mu$.  Works in the literature, starting from Brian\c{c}on and Lamboley \cite{BrianconLamboley} have shown almost minimality properties for the hard constraint minimizers but the arguments we are aware of seem to not be sufficiently quantitative to adapt to the homogenization theory setting.  

We take a different approach which exploits an almost dilation invariance property of $J_\mu$ minimizers at large scales.  We do not show an almost minimality property for volume constrained minimizers.  Although there may be singular values of the volume $m \geq 1$ which are not achieved by $J_\mu$ minimizers, we show that all volume constrained minimizers have small energy deficit $J_\mu - \inf J_\mu$ for a well chosen $\mu$. Once we know the $J_\mu$ energy deficit is small we use a selection principle / penalization type argument to find a $J_g$ minimizer near the volume constrained minimizer for which the results of \sref{augmented-reg} apply.  

Recall we defined the volume constrained energy deficit
\[ \delta_1(U,a) = |U|^{\frac{2}{d}}(\lambda_1(U,a) - \inf_{|V| = |U|} \lambda_1(V,a)).\]

  \begin{proposition}\label{p.hard-constraint-replacement}
For every volume $m \geq 2$ and $U$ with $|U| = m$ such that $\delta_1(U,a) \leq \frac{1}{2}$
 there is $\mu_* >0$ so that for all $p>d+4$ there is a function $g$ satisfying the hypotheses of all parts of \tref{augmented-full} (relative to this $\mu_*$) so that the $J_g$ minimizer $\Omega$ has
\begin{equation}\label{e.eigenvalue-closeness-Omega}
 m^{\frac{2}{d}}|\lambda_1(U,a) - \lambda_1(\Omega,a)| \leq C\left[\delta_1(U,a) |\log \delta_1(U,a)|^{p} + m^{-\frac{1}{d}}|\log (2+m)|^{p}\right]
 \end{equation}
and
\begin{equation}\label{e.measure-closeness-Omega}
 m^{-1}|U \Delta \Omega| \leq C\left[\delta_1(U,a) |\log \delta_1(U,a)|^{p} + m^{-\frac{1}{d}}|\log (2+m)|^{p}\right]
 \end{equation}
with $C\geq 1$ depending on the universal constants $(\Lambda,d,\|\grad a\|_\infty)$ and on $p$.
 
  \end{proposition}
  \begin{remark}
  By \tref{augmented-full} $\Omega_*$ is a (large scale) Lipschitz domain so the (almost) linear measure estimate on $m^{-1}|U_* \Delta \Omega_*|$ is a nontrivial regularity property on $U_*$. It is like large-scale Lipschitz domain property but measured in $L^1$ instead of $L^\infty$ (and with logarithmic loss).  In particular also note that this estimate is stronger the square root rate given by the eigenvalue estimate plus Faber-Krahn optimal stability.  
  \end{remark}
  
  \begin{remark}
 The result \pref{hard-constraint-replacement} does not actually require us to know existence of minimizers for \eref{constrained}.
  \end{remark}
  
  From \pref{hard-constraint-replacement} we can derive our main result about volume constrained minimizers \tref{main}. The proof can be found below in \sref{main-proof}.

We outline the section.  First in \sref{volume-map} we consider the values of the volume taken by $J_\mu$ minimizers over $\mu \in (0,1]$.  This map is monotone but may have jumps where certain volumes are missed.  In \sref{singular-mu} at these singular values of $\mu$ we use a dilation and convexity argument to show that volume constrained minimizers with volumes inside of the jump have close to minimal value of $J_\mu$.  Then in \sref{selection-principle} we use a selection principle / penalization argument to show that when $U$ has close to minimal $J_\mu$ energy there is a $J_g$ minimizer nearby.

\subsection{The $\mu$ to volume map and its singular values}\label{s.volume-map}First we note that augmented minimizers are volume constrained minimizers with their given volume.
\begin{lemma}\label{l.mutovol}
 Fix $\mu>0$.  If $U_*$ is a domain minimizer of the augmented functional $J_\mu$ \eref{augmented} then $U_*$ minimizes $\lambda_1$ with volume constraint $m = |U_*|$ \eref{constrained}.  Furthermore if $V$ is another constrained minimizer of \eref{constrained} with volume $m$ then $V$ also minimizes $J_\mu$ \eref{augmented}.
\end{lemma}
\begin{proof}
Suppose that $|V| = m = |U_*|$ then
\[ \lambda_1(U_*,a) + \mu m \leq \lambda_1(V,a) + \mu m\]
i.e. $ \lambda_1(U_*,a) \leq \lambda_1(V,a)$.

On the other hand suppose that $V$ is a minimizer for \eref{constrained} with volume $m = |U_*|$.  Then 
\[ J_\mu(V) = \lambda_1(V,a) + \mu m \leq \lambda_1(U_*,a) + \mu m = J_\mu(U_*)\]
so $V$ also minimizes $J_\mu$ as in \eref{augmented}.
\end{proof}

This relationship creates a mapping from $\mu \in (0,\infty)$ to the volumes (possibly non-unique) created by the augmented minimization problem \eref{augmented}
\begin{equation}
 \textup{Vol}(\mu,a) := \{|U| :  \ \hbox{ $U$ is a minimizer of \eref{augmented}} \}.
 \end{equation}
 also define
 \[ \textup{Vol}(a) = \cup_{\mu>0} \textup{Vol}(\mu,a).\]
 \begin{remark}\label{r.singletons}
 By \lref{mutovol} if $\textup{Vol}(\mu_0,a)$ is a singleton then the problem \eref{augmented} with $\mu = \mu_0$ is equivalent to \eref{constrained} with $m = |U_0|$ where $U_0$ is any $J_{\mu_0}$ minimizer.
 \end{remark}
Next we check that the set-valued operator $\textup{Vol}$ is monotone decreasing with respect to $\mu$ in a certain sense.
\begin{lemma}\label{l.volmumon}
For any $\mu_1 < \mu_2$
\[ \sup \textup{Vol}(\mu_2,a) \leq \inf \textup{Vol}(\mu_1,a). \]
\end{lemma}
\begin{proof}
Suppose that $U_2$ is a minimizer for $J_{\mu_2}$ and $U_1$ is a minimizer for $J_{\mu_1}$.  Then
\begin{align*}
 \lambda_1(U_2,a) + \mu_2 |U_2| &\leq \lambda_1(U_1,a) + \mu_2 |U_1| \\
 &\leq \lambda_1(U_1,a) + \mu_1 |U_1| +(\mu_2 - \mu_1)|U_1|\\
 & \leq  \lambda_1(U_2,a) + \mu_1 |U_2|+(\mu_2 - \mu_1)|U_1|
 \end{align*}
 so
 \[(\mu_2-\mu_1) |U_2| \leq (\mu_2 - \mu_1)|U_1|. \]
 Since $\mu_2 > \mu_1$ we can conclude.
\end{proof}

The best case scenario is that the map $\textup{Vol}(\mu,a)$ is surjective, in the sense that $\textup{Vol}(a) = \cup_{\mu>0} \textup{Vol}(\mu,a) = (0,\infty)$.  It seems quite possible this scenario does not occur.

By \rref{singletons} the issue is entirely in the ``singular" values
\begin{equation}
 B(a) = \{ \mu \in (0,\infty): \hbox{ $\textup{Vol}(\mu,a)$ is not a singleton}\}.
 \end{equation}
By \lref{volmumon} the intervals
\[ (\inf\textup{Vol}(\mu,a), \sup\textup{Vol}(\mu,a)) \ \hbox{ are disjoint for distinct $\mu$}\]
because this open interval contains a rational for each $\mu \in B(a)$ this implies that $B(a)$ is countable.  Although the singular values of $\mu$ are countable the values which are missed by $\textup{Vol}(a)$ may not be.

We also conclude with one more piece of abstract information that can be derived without further specifics from our homogenization setting: the extremal volumes in each jump of $\textup{Vol}(\cdot,a)$ are obtained.  Actually this follows from the stronger statement:
\begin{lemma}\label{l.volsetclosed}
The set $\textup{Vol}(a) \subset (0,\infty)$ is (relatively) closed.
\end{lemma}

The proof is postponed momentarily, it uses some compactness notions typical for these shape optimization problems.

\begin{corollary}\label{c.extremalsarein}
For $\mu>0$ call 
\[\textup{Vol}_\pm(\mu) = \lim_{\mu' \to \mu\mp} \textup{Vol}(\mu')\]
  There are volume constrained minimizers $U_\pm(\mu)$ with volumes $|U_\pm(\mu)| = \textup{Vol}_\pm(\mu)$.
\end{corollary}
Note that the left/right limits in the statement exist due to monotonicity.
\begin{proof}
  There is a sequence of volumes $m_j \searrow \textup{Vol}_+(\mu)$ and $\mu_j \nearrow \mu$ with $m_j \in \textup{Vol}(\mu_j,a)$.  By \lref{volsetclosed} $m \in \textup{Vol}(a)$.  The other side is symmetric. 
\end{proof}

\begin{proof}[Proof of \lref{volsetclosed}]
Let $m \in \overline{\textup{Vol}(a)} \cap (0,\infty)$ so there is a sequence of $J_{\mu_j}$ minimizers $U_j$ with $|U_j| \to m$.  By \lref{scaling-bounds} the sequece $\mu_j$ is bounded, we can assume without loss that it converges to some $\mu$.  Let $u_j = u_{U_j} \in H^1_0(U_j)$ the corresponding principal eigenfunctions.  By the same argument as in the proof of \tref{global-existence}, up to a finite number of period translations of disjoint components of each $U_j$ (which do not affect the energy or the volume $|U_j|$), we can assume that the $U_j$ are all contained in some fixed compact region $K$.

By \tref{augmented-full} the $u_j$ are uniformly bounded in $H^1_0(K)$ norm, uniformly bounded in supremum norm, uniformly Lipschitz continuous, uniformly non-degenerate at their free boundary, uniform inner and outer density estimates, and have uniform bound on their perimeter in the sense of \dref{perimeter-hyp}.  Thus, up to taking a subsequence, the $u_j$ converge weakly in $H^1$ and uniformly to some $u_\infty$ which is also Lipschitz continuous.  In particular $U_\infty:=\{u_\infty>0\}$ is an open set and immediately (from the weak $H^1$ convergence)
\[ \lambda_1(U_\infty,a) \leq \int_{U_\infty} a(x) \grad u_\infty \cdot \grad u_\infty \ dx \leq \liminf_{j \to \infty} \lambda_1(U_j,a).\]

Now lets assume, temporarily, that $ |U_\infty| = \lim_{j \to \infty} |U_j|$ and conclude. So if $V$ is another bounded quasi-open set then
\[ J_\mu(U_\infty,a) \leq \liminf_{j \to \infty} J_{\mu_j}(U_j,a)\leq \liminf_{j \to \infty} J_{\mu_j}(V,a) = J_{\mu}(V,a)\]
and so $U_\infty$ is a $J_\mu$ minimizer and so $m = |U_\infty| \in \textup{Vol}(a)$.

We need to establish that $ |U_\infty| = \lim_{j \to \infty} |U_j|$.  This follows from dominated convergence theorem if we can show the pointwise a.e. convergence of the indicator functions.  

For $x \in U_\infty$ we have $x \in U_j$ for $j$ sufficiently large by the uniform convergence.  For $x$ an interior point of the complement $\R^d \setminus U_\infty$, i.e. $r = d(x,\overline{U}_\infty)>0$, we claim the same holds, $x \in \R^d \setminus \overline{U}_j$ for $j$ sufficiently large.  If not there is a sequence $x_j \in \overline{U}_j$ with $x_j \to x$.  By the uniform non-degeneracy of the $u_j$ we have $\inf_j\sup_{y\in B_{r/2}(x_j)} u_j(y) >0$. For $j$ sufficiently large this contradicts $u_\infty \equiv 0$ on $B_r(x)$.  

Thus ${\bf 1}_{U_j} \to {\bf 1}_{U_\infty}$ pointwise on $\R^d \setminus \partial U_\infty$.  All that is left is to show that $\partial U_\infty$ has zero measure.  Let $x \in \partial U_\infty$, by the arguments in the previous paragraph there are sequences of points $x_j \in U_j$ and $y_j \in \R^d \setminus \overline{U}_j$ converging to $x$.  Hence there is also a sequence $z_j \in \partial U_j$ converging to $x$.  Thus if $\{B_r(x_k)\}_{k=1}^N$ is a finite overlapping covering of $\partial U_\infty$ then there are $z_{k,j} \in \partial U_j \cap B_{r/2}(x_k)$ for $j$ sufficiently large and the collection $\{B_{r/2}(z_{j,k})\}_{k=1}^N$ is also a finite overlapping collection of balls centered at points of $\partial U_j$ so (by the uniform inner density estimates for $U_j$ and the uniform boundary strip area bounds)
\[  \sum_k |B_r(x_k)| \leq C\sum_k |U_j \cap B_{r/2}(z_{j,k})| \leq C|\{x \in U_j: d(x, \partial U_j) \leq \frac{r}{2}\}| \leq Cr.\]
Since $r>0$ was arbitrary $\partial U_\infty$ has measure zero, and we have finally shown that ${\bf 1}_{U_j} \to {\bf 1}_{U_\infty} $ pointwise almost everywhere.

\end{proof}

\subsection{The $\mu$ to volume map for approximately dilation invariant operators}\label{s.singular-mu}  We are aiming for a result even at the singular values of the volume constraint $B(a)$.

In the homogenization setting we lack exact dilation invariance, but we nonetheless have some approximate dilation invariance coming from the large scale limit.

  \begin{lemma}[Approximate scaling relation]\label{l.approx-scaling}
 Suppose that $U$ has $|U| \geq 1$ and satisfies the following properties
 \begin{enumerate}[label = (\roman*)]
 \item $-\grad \cdot( a(x) \grad \cdot)$-Eigenfunction $L$-Lipschitz estimate \dref{L-Lipschitz}.
  \item $-\grad \cdot( a(x) \grad \cdot)$-Eigenfunction $\ell$-non-degeneracy estimate \dref{ell-non-degen}.
 \item Domain $(r_0,M,h)$-(large scale)Lipschitz estimate \dref{large-scale-lipschitz}.
 \end{enumerate}
 Then there is $C\geq 1$ depending on universal parameters and on $(r_0,M,h,L,\ell)$ so that
\begin{equation}\label{e.approx-scaling}
 t^2|U|^{\frac{2}{d}}(\lambda_1(tU,a) - t^{-2}\lambda_1(U,a)) \leq  C t^{-1}|U|^{-\frac{1}{d}}|\log (2+|tU|)|^{\frac{1}{2}+\eta} \ \hbox{ for all } \ t \geq |U|^{-\frac{1}{d}}.
 \end{equation}
 \end{lemma}
 \begin{proof}
 By the hypothesis there is an $(r_0,M)$ - Lipschitz domain $U_- \subset U$ such that $d_H(\partial U_-,\partial U) \leq h$. Using domain monotonicity of eigenvalues 
\[  t^{-2}\lambda_1(U,a) \geq  t^{-2}\lambda_1(U_-,a)\]
and by \lref{interior-approx}
\[ t^2|U|^{\frac{2}{d}}(\lambda_1(tU,a) - \lambda_1(tU_-,a)) \leq Ct^{-1}|U|^{-\frac{1}{d}} \]
so
\[ t^2|U|^{\frac{2}{d}}(\lambda_1(tU,a) - t^{-2}\lambda_1(U,a)) \leq t^2|U|^{\frac{2}{d}}(\lambda_1(tU_-,a) - t^{-2}\lambda_1(U_-,a)) +Ct^{-1}|U|^{-\frac{1}{d}}.\]
Then apply \cref{optimal-eigen-conv} to both terms on the right of
\[ \lambda_1(tU_-,a) - t^{-2}\lambda_1(U_-,a) =  [\lambda_1(tU_-,a) - \lambda_1(tU_-,\bar{a})]+[t^{-2}\lambda_1(U_-,\bar{a}) - t^{-2}\lambda_1(U_-,a)]\]
using the exact scaling invariance of $\lambda_1(\cdot,\bar{a})$.
 \end{proof}

 Now the main result which motivates all the above is that a $\sigma$ approximate dilation invariance property for $J_\mu$ minimizers implies that volume constrained minimizers at \emph{all} volumes are $\sigma$-close to being an augmented minimizer.  {\bf Note:} We do \emph{not} have a good approximate dilation invariance property for volume constrained minimizers, the trick is we can use the good property of $J_\mu$ minimizers.

\begin{proposition}\label{p.sigmaerror}
For any $m \in [\textup{Vol}_-(\mu),\textup{Vol}_+(\mu)]$
\[ m^{\frac{2}{d}}\bigg(\inf_{|V| = m} \lambda_1(V,a) + \mu m - \inf_V J_\mu(V,a)\bigg) \leq  C m^{-\frac{1}{d}}|\log (2+m)|^{\frac{1}{2}+\eta} . \]
\end{proposition}

\begin{proof}

The only interesting case is when $\mu \in B(a)$ so, by \cref{extremalsarein}, there are two distinct $J_\mu(\cdot,a)$ minimizers $U_1$ and $U_2$ with 
\[ |U_1| = \textup{Vol}_-(\mu) < \textup{Vol}_+(\mu)=|U_2|.\]
Furthermore these minimizers satisfy the approximate dilation invariance of \lref{approx-scaling}, given the regularity theory in \tref{augmented-full}.  Note that the right hand side of the inequality \eref{approx-scaling} is bounded by $C m^{-\frac{1}{d}}|\log (2+m)|^{\frac{1}{2}+\eta}:=\sigma$ as long as $t \sim_{\Lambda,d} 1$.

 Call $T = (|U_2|/|U_1|)^{1/d} \sim_{d,\Lambda} 1$.  Then, since $U_2$ is a volume constrained minimizer with volume $|U_2|$,
\begin{align*}
\lambda_1(U_2,a) &\leq \lambda_1(TU_1,a) \\
&\leq T^{-2}\lambda_1(U_1,a)+  \sigma T^{-2}|U_1|^{-\frac{2}{d}}\\
&=T^{-2}\lambda_1(U_1,a)+  \sigma|U_2|^{-\frac{2}{d}}
 \end{align*}
 and, symmetrically,
 \begin{align*}
 \lambda_1(U_1,a)  \leq \lambda_1(T^{-1}U_2,a)  \leq T^{2}\lambda_1(U_2,a)  +  \sigma |U_1|^{-\frac{2}{d}}
 \end{align*}
 so
 \[ T^{-2}\lambda_1(U_1,a) + \mu T^d|U_1| \leq J_\mu(U_2,a) +  \sigma |U_2|^{-\frac{2}{d}}.\]

Now suppose that $U_*$ is a volume constrained eigenvalue minimizer of \eref{constrained} for some $|U_1| < m < |U_2|$.  This is for convenience, we don't really need existence here, we could just choose an approximating sequence. Call $t = (m/|U_1|)^{1/d}$ then
\begin{align*}
 J_\mu(U_*,a)  &= \lambda_1(U_*,a) +\mu m \\
 &\leq \lambda_1(tU_1,a) + \mu m \\
 &\leq t^{-2}[\lambda_1(U_1,a)+\sigma|U_1|^{-\frac{2}{d}}] + \mu t^{d}|U_1| \\
 &=: f(t)
 \end{align*}
 by convexity of this function $f(t)$
 \begin{align*}
  J_\mu(U_*,a) &\leq \max\{f(1),f(T)\}\\
  & = \max\bigg\{J_\mu(U_1,a)+\sigma|U_1|^{-\frac{2}{d}},T^{-2}\lambda_1(U_1,a) + \mu T^{d}|U_1|+\sigma|U_2|^{-\frac{2}{d}}\bigg\} \\
  &\leq\max\bigg\{J_\mu(U_1,a),T^{-2}\lambda_1(U_1,a)+ \mu T^{d}|U_1|\bigg\} +  \sigma |U_1|^{-\frac{2}{d}} \\
  &\leq\max\bigg\{J_\mu(U_1,a),J_\mu(U_2,a)\bigg\}+2 \sigma |U_1|^{-\frac{2}{d}}.
  \end{align*}
  Since, by \lref{scaling-bounds} every minimizer of $J_\mu$ has $|U| \sim_{d,\Lambda} \mu^{-\frac{d}{d+2}}$ we can bound $|U_1|^{-\frac{2}{d}} \leq C(d,\Lambda)m^{-\frac{2}{d}}$.
  \end{proof}

\subsection{Selection principle / penalized minimization problem for the singular volumes}\label{s.selection-principle}
Consider a volume constrained minimizer $U_*$ of \eref{constrained}, or even just a set with small deficit $\delta_1(U_*,a) \leq \frac{1}{2}$.   we have shown above that if $\mu_*>0$ is the value such that $|U_*| \in [\textup{Vol}_-(\mu_*),\textup{Vol}_+(\mu_*)]$ then
\[ |U_*|^{\frac{2}{d}}(J_{\mu_*}(U_*,a) - \min J_{\mu_*}(\cdot,a)) \leq C |U_*|^{-\frac{1}{d}}|\log (2+|U_*|)|^{\frac{1}{2}+\eta},\]
i.e. the energy deficit for $J_{\mu_*}$ is small.  Now we will introduce a regularization of $U_*$ which is a minimizer of a function $J_g$ of the type considered in \sref{augmented-reg}.  The idea is inspired by Brasco, De Philippis, and Velichkov \cite{BDPV} which itself was inspired by an idea from Cicalese and Leonardi \cite{CicaleseLeonardi} studying the quantitative isoperimetric inequality.  

Call $\rho(x,V)$ to be the signed distance function the boundary of $V$ positive outside of $V$. Let $\omega:[0,\infty) \to [0,\infty)$ be modulus of continuity to be specified, and extend $\omega$ to $\R$ by odd reflection.  We define
\[ g(x) = \mu_*\left[1 + \omega\left(\frac{\rho(x, U_*)}{|U_*|^{\frac{1}{d}}}\right)\right].\]
Then consider the augmented functional
\begin{equation}\label{e.penalized-functional}
 J_g(\Omega,U_*,a) = \lambda_1(\Omega,a) + \int_{\Omega} g(x) \ dx
 \end{equation}
We first make a remark, see the simple computation in \sref{distJgcompute} below, that
\begin{equation}\label{e.Jgidentity}
 J_g(\Omega,U_*,a) = J_{\mu_*}(\Omega,a) + \mu_*\textup{dist}_\omega(\Omega,U_*)- \mu_*\int_{U_*} \omega\left(\frac{d(x,\partial U_*)}{|U_*|^{\frac{1}{d}}}\right) \ dx
 \end{equation}
where the last term is a constant with respect to $\Omega$, so does not affect the minimization, and
\[ \textup{dist}_\beta(\Omega,U_*):=\int_{\Omega \Delta U_*} \omega\left(\frac{d(x,\partial U_*)}{|U_*|^{\frac{1}{d}}}\right) \ dx.\]
Thus, philosophically speaking, when we minimize $J_g$ we are making a selection of an $\Omega_*$ which balances minimizing $J_{\mu_*}$ with being nearby $U_*$. It is a typical distance penalized minimization problem.

\begin{remark}
For intuition think of $\omega(s) = s^\beta$ a H\"older modulus.  We expect, and prove below, that $\textup{dist}_\beta(\Omega,U_*)$ controls $|\Omega \Delta U_*|^{1+\beta}$ (up to some scaling details).  Thus we would prefer to choose $\beta$ as small as possible, ideally $\beta = 0$, but we still need some amount of continuity to use the regularity theory of $J_g$.  Specifically we need $\omega$ to be summable over geometric sequences of scales so $\omega(s) = |\log(s)|^{-1 -\delta}$ works for any $\delta>0$. Note that we lose a similar (better by a factor of $\frac{1}{2}$) logarithmic factor anyway from \tref{optimal-hom-lip} and we are not currently attempting to optimize the logarithmic factors.
\end{remark}

First we give the relation between $\textup{dist}_\omega$ and the measure difference.
\begin{lemma}\label{l.dist-omega-est}
Suppose $A$ and $B$ are sets in $\R^d$ and $B$ satisfies the boundary strip area bound \dref{perimeter-hyp} with constant $P$, then
\[ \frac{|A \Delta B|}{|B|} \leq 2P^{\frac{\beta}{1+\beta}}\left(\frac{\textup{dist}_\beta(A,B)}{|B|}\right)^{\frac{1}{1+\beta}}.\]
  For $\omega(r) = |\log(2+r^{-1})|^{-p} \wedge \gamma $, for some $\gamma >0$,
\[ \frac{|A \Delta B|}{|B|} \leq \frac{\textup{dist}_{\omega}(A,B)}{|B|}\left[1+\gamma^{-1}+\left|\log \left(2+\frac{P|B|}{\textup{dist}_{\omega}(A,B)}\right)\right|^{p}\right].\]
\end{lemma}
The second choice of $\omega$ is the one we will use, the first is meant as an example to explain the intuition. The proof is postponed to \sref{dist-omega-est}, it is a standard idea with co-area formula.

  Next we show how the penalized minimization property for $J_g$ does not perturb $U$ or $\lambda_1(U,a)$ too much.

\begin{lemma}\label{l.penalized-functional-consequences}
If 
\[ |U|^{\frac{2}{d}}(J_\mu(U,a) - \inf J_\mu(\cdot,a)) \leq \sigma \leq 1 \]
 and $\Omega$ minimizes $J_g(\cdot,U,a)$ over quasi-open sets then
\[\frac{\textup{dist}_\omega(\Omega,U)}{|U|} \leq C \sigma \ \hbox{ and } \  |U|^{\frac{2}{d}}|\lambda_1(\Omega,a)  - \lambda_1(U,a)| \leq 2\sigma + C\frac{|\Omega \Delta U|}{|U|}\]
for some universal $C \geq 1$.
\end{lemma}
So if we can prove a convergence estimate of $\lambda_1(\Omega)$ and $\Omega$ then we can do the same for $U$ up to an additional $\sigma|\log \sigma|^{p}$ error.

\begin{proof}
Energy computation with the formula \eref{Jgidentity} and \pref{sigmaerror} gives
\[ \inf J_{\mu}(\cdot,a) \leq J_{\mu}(\Omega,a)\leq J_{\mu}(\Omega,a)+\mu\, \textup{dist}_\omega(\Omega,U) \leq J_{\mu}(U,a) \leq \inf J_{\mu}(\cdot,a) + \sigma|U|^{-\frac{2}{d}}\]
so
\[ |U|^{\frac{2}{d}}(J_{\mu^*}(\Omega,a)- \min J_{\mu}(\cdot,a)) \leq  \sigma\]
and
\[ \mu|U|^{\frac{2}{d}}\textup{dist}_\omega(\Omega,U) \leq  \sigma\]
or, by the scalings, in \lref{scaling-bounds} $\mu\sim_{\Lambda,d} |U|^{-1 - \frac{2}{d}}$ so
\[ \frac{\textup{dist}_\omega(\Omega,U)}{|U|} \leq C \sigma. \]
We can also conclude from the above inequalities
\begin{align*}
 (\lambda_1(\Omega,a)  - \lambda_1(U,a)) &\leq J_{\mu}(\Omega,a)+\mu\, \textup{dist}_\omega(\Omega,U) - J_{\mu}(U,a) - \mu(|\Omega| - |U|)\\
 & \leq C\sigma |U|^{-\frac{2}{d}} +\mu(|U| - |\Omega|) 
 \end{align*}
 and
 \begin{align*}
 (\lambda_1(U,a)  - \lambda_1(\Omega,a)) &= J_{\mu}(U,a) - J_{\mu}(\Omega,a) + \mu(|\Omega| - |U|)\\
 & \leq \sigma |U|^{-\frac{2}{d}} +\mu\, \textup{dist}_\omega(\Omega,U) +\mu(|\Omega| - |U|)\\
 &\leq 2\sigma |U|^{-\frac{2}{d}}+\mu(|\Omega| - |U|)
 \end{align*}
 the conclusion follows from using the scaling relations \lref{scaling-bounds} again.

\end{proof}

\subsection{Proof of \pref{hard-constraint-replacement}}

At this point all the significant elements are in place, they simply need to be put together.  As in the statement let $U$ be a domain with volume $|U| \geq 2$ and energy deficit
\[ \delta_1(U,a) = |U|^{\frac{2}{d}}(\lambda_1(U,a) - \inf_{|V| = |U|} \lambda_1(V,a)) \leq \frac{1}{2}.\]
By \pref{sigmaerror} there is $\mu_*>0$ with $\mu_* \sim_{\Lambda,d} |U|^{-1-\frac{2}{d}}$ such that $|U| \in [\textup{Vol}_-(\mu_*),\textup{Vol}_+(\mu_*)]$ and 
\[ |U|^{\frac{2}{d}}(J_\mu(U,a) - \inf_{V} J_\mu(V,a)) \leq \delta_1(U,a) + C|U|^{-\frac{1}{d}}|\log |U||^{\frac{1}{2}+\eta}\]
and we define
\begin{equation}\label{e.sigma-choice-conv}
 \sigma = \delta_1(U,a) + |U|^{-\frac{1}{d}}|\log |U||^{\frac{1}{2}+\eta}
 \end{equation}
so we do not need to write that repeatedly.  

Now, given $\eta>0$, fix 
\[ g(x)  = \mu_*\left[1+\gamma_0 \wedge \omega\left(\frac{d(x,\partial U_*)}{|U|^{\frac{1}{d}}}\right)\right] \ \hbox{ with } \  \omega(s) = |\log (2+s^{-1})|^{-p}\]
where $\gamma_0$ is sufficiently small depending on $(\Lambda,d,\|\grad a\|_\infty,\eta)$ so that \tref{augmented-full} parts \partref{4} and \partref{5} hold. Note that $g$ satisfies the hypotheses \eref{g-hyp1}, \eref{g-hyp2} (since $p>d+4$) and \eref{g-hyp3}.  In particular note that the localizing hypothesis for existence \eref{g-hyp3} holds since $U_*$ is a bounded domain and $\omega(s) \to |\log 2|^{-p}$ as $s \to \infty$ and we can assume without loss that this is strictly larger than $\gamma_0$.

Let $\Omega_*$ be a minimizer of the penalized functional $J_g$.  This domain $\Omega_*$ satisfies all the regularity properties given in \tref{augmented-full}, Lipschitz eigenfunction, non-degenerate eigenfunction, and $\Omega_*$ is a (large scale) Lipschitz domain.  All of the parameters depend only on universal constants and $p$.

By \lref{penalized-functional-consequences}
\[\frac{\textup{dist}_\omega(\Omega_*,U)}{|U|} \leq C \sigma \ \hbox{ and } \  |U|^{\frac{2}{d}}|\lambda_1(\Omega_*,a)  - \lambda_1(U,a)| \leq 2\sigma + C\frac{|\Omega \Delta U|}{|U|}\]
then by \lref{dist-omega-est}
\[ \frac{|\Omega_* \Delta U|}{|U|} \leq C\frac{\textup{dist}_\omega(\Omega_*,U)}{|\Omega_*|} \left|\log \left[c\frac{\textup{dist}_\omega(\Omega_*,U)}{|\Omega_*|} \right] \right|^{p} \leq C\sigma \left|\log \left[c\sigma \right] \right|^{p}\]
and so also
\[ |U|^{\frac{2}{d}}|\lambda_1(\Omega_*,a)  - \lambda_1(U,a)| \leq C\sigma \left|\log \left[c\sigma \right] \right|^{p}.\]
Note that the double logarithmic terms that appear when compute $|\log \sigma|$ can be absorbed by just increasing the power of the logarithm slightly.  

\qed

\subsection{Proof of \tref{main}}\label{s.main-proof}  Let $U$ as in the statement of the theorem and let $\Omega$ be the replacement domain from \pref{hard-constraint-replacement} which is a large scale Lipschitz domain and has \eref{eigenvalue-closeness-Omega} and \eref{measure-closeness-Omega}. Let $E$ be an $\bar{a}$-ellipsoid with volume $|E| = |U|$
\begin{align*}
 |U|^{\frac{2}{d}}\lambda_1(U,a) &\leq |U|^{\frac{2}{d}}\lambda_1(E,a) + \delta_1(U,a) \\
 &\leq |U|^{\frac{2}{d}}\lambda_1(E,\bar{a}) + C|U|^{-\frac{1}{d}}|\log(2+|U|)|^{\frac{1}{2}+\eta}+\delta_1(U,a) \\
 &\leq |U|^{\frac{2}{d}}\lambda_1(\Omega,\bar{a}) + C|U|^{-\frac{1}{d}}|\log(2+|U|)|^{\frac{1}{2}+\eta}+\delta_1(U,a) \\
&\leq |U|^{\frac{2}{d}}\lambda_1(\Omega,a) + C|U|^{-\frac{1}{d}}|\log(2+|U|)|^{\frac{1}{2}+\eta}+\delta_1(U,a)\\
& \leq |U|^{\frac{2}{d}}\lambda_1(U,a) + C|U|^{-\frac{1}{d}}|\log(2+|U|)|^{p}+C\delta_1(U,a)|\log\delta_1(U,a)|^{p}
 \end{align*}
where the second inequality is by \cref{optimal-eigen-conv}, the fourth is by \cref{quanthom-large-scale-lip}, and the last is by \eref{eigenvalue-closeness-Omega} from \pref{hard-constraint-replacement}.

Note that we do \emph{not} directly have any information on $\lambda_1(U,\bar{a})$, however we do have from the above chain of inequalities,
\[ |U|^{\frac{2}{d}}(\lambda_1(\Omega,\bar{a}) - \lambda_1(E,\bar{a})) \leq C|U|^{-\frac{1}{d}}|\log(2+|U|)|^{p}+C\delta_1(U,a)|\log\delta_1(U,a)|^{p}\]
and so we can apply the optimal Faber-Krahn stability, \tref{FaberKrahn}, to $\Omega$
\[ \inf_{E}\frac{|\Omega \Delta E|}{|E|} \leq C|U|^{-\frac{1}{2d}}|\log(2+|U|)|^{\frac{p}{2}}+C\delta_1(U,a)^{\frac{1}{2}}|\log\delta_1(U,a)|^{\frac{p}{2}}\]
and then using \eref{measure-closeness-Omega} from \pref{hard-constraint-replacement} to bound $|U \Delta E|$ we get the same estimate for the asymmetry of $U$ concluding the proof. \qed

\appendix
\section{Computations}
\subsection{Eigenvalues optimizers}  The domain $U$ with volume $|U| = \omega_d\rho^d$ minimizing $\lambda(U,\textup{id})$ is known to be (any translate of) the ball of radius $\rho$.  For any constant coefficient elliptic operator we have the symmetry
 \[ \lambda_1(U,\bar{a}) = \lambda_1(\bar{a}^{-1/2}U,\textup{id}).\]
 To check this, given $u$ on $U$ with $\|u\|_{L^2(U)} = 1$ we can define $v(x) = \textup{det}(\bar{a})^{\frac{1}{4}}u(\bar{a}^{1/2}x)$ which has $L^2$ norm $1$ on $x\in \bar{a}^{-1/2} U$ so 
 \begin{align*}
  \int_{\bar{a}^{-1/2} U} \grad v \cdot \grad v &= \int_{\bar{a}^{-1/2} U} \bar{a}^{1/2}\grad u(\bar{a}^{1/2}x) \cdot \bar{a}^{1/2}\grad u(\bar{a}^{1/2}x) \ \textup{det}(\bar{a})^{\frac{1}{2}}dx \\
  &= \int_{U} \bar{a} \grad u \cdot \grad u \ dy  
  \end{align*}
 with $y = \bar{a}^{1/2}x$ and $dy = \textup{det}(\bar{a})^{1/2}dx$. 
 
 So the shape optimizers for $\lambda_1(\cdot,\bar{a})$ are ellipsoids of the form
 \[ E = \bar{a}^{1/2}B. \]
 \subsection{$J_\mu$ minimizers}\label{s.Jmu-minimizers-computations} The domain minimizing
 \[ J_\mu(U,\bar{a}) = \lambda_1(U,\bar{a}) + \mu |U| = \lambda_1(\bar{a}^{1/2}U,\textup{id}) + \mu|U|\]
 will certainly minimize $\lambda_1(U,\bar{a})$ over all domains with its volume so it is of the form
 \[ U = \bar{a}^{1/2}B_{\rho}\]
 for some $\rho>0$.  Then the functional is one-dimensional so we just minimize
 \begin{align*}
  f(\rho) &:= J_\mu(\bar{a}^{1/2}B_{\rho},\bar{a}) \\
  &= \lambda_1(B_\rho,\textup{id}) + \mu  \ \textup{det}(\bar{a})^{1/2}|B_1| \rho^d \\
  &= \rho^{-2} \lambda_1(B_1,\textup{id})+ \mu  \ \textup{det}(\bar{a})^{1/2}|B_1| \rho^d
  \end{align*}
  and this function is minimized when
  \begin{equation}
   \rho^{d+2} = \tfrac{2}{d}\mu^{-1}\textup{det}(\bar{a})^{-\frac{1}{2}}|B_1|^{-1}\lambda_1(B_1,\textup{id}).
   \end{equation}
  In particular note that when $\bar{a} = \textup{id}$ then $B_1$ minimizes $J_\mu(\cdot,\textup{id})$ when
 \begin{equation}\label{e.ball-eigenfunction-slope}
 \mu = \tfrac{2}{d|B_1|}\lambda_1(B_1,\textup{id}) \ \hbox{ implying } \ \left.|\grad u_{B_1}|\right|_{\partial B_1} = \sqrt{\tfrac{2}{d|B_1|}\lambda_1(B_1,\textup{id})}
 \end{equation}
In general
 \[ \rho \sim \mu^{-\frac{1}{d+2}} \ \hbox{ and } \ |U| = |B_1|\rho^d \sim \mu^{-\frac{d}{d+2}}.\]
 Next we compute the same scaling bounds for a general elliptic coefficient field.
 \begin{lemma}\label{l.scaling-bounds}
 Suppose $a$ is uniformly elliptic and $U$ minimizes $J_\mu(\cdot,a)$ over quasi-open sets in either $\R^d$ or $\R^d /N\Z^d$ when $N$ is an integer $N \geq \mu^{-\frac{1}{d+2}}$. Then
 \[ \lambda_1(U,a) \sim_{d,\Lambda} \mu^{\frac{2}{d+2}} \ \hbox{ and } \ |U| \sim_{d,\Lambda} \mu^{-\frac{d}{d+2}}.\]
 The same holds for $J_g$ minimizers for $g$ satisfying the bounds $(1+\gamma)^{-1} \leq g/\mu \leq 1+\gamma$ with constants depending additionally on $\gamma$.
 \end{lemma}
 \begin{proof}
 Compare with $B_\rho$
\[\lambda_1(U,a) + \mu |U| \leq \lambda_1(B_\rho,a) + \mu |B_\rho| \leq \lambda_1(B_\rho, \Lambda \textup{id})+\mu |B_\rho| \leq C(d,\Lambda)\rho^{-2} + \mu |B_\rho|\]
Then choose $\rho =  \mu^{-\frac{1}{d+2}}$, and notice that $N \geq \rho$ so there is a ball $B_\rho$ in the torus $\R^d /N\Z^d$, so
\[ \lambda_1(U,a) +\mu|U|\leq C(d,\Lambda)\mu^{\frac{2}{d+2}} .\]
On the other side with $r$ such that $|B_r| = |U|$
\[ \lambda_1(U,a) \geq \lambda_1(U,\Lambda^{-1}\textup{id}) \geq \lambda_1(B_{r},\Lambda^{-1}\textup{id}) \geq c(d,\Lambda)|U|^{-2/d}\]
and so, using the upper bounds for $|U|$ and $\lambda_1(U,a)$ respectively,
\[ \lambda_1(U,a) \geq c(d,\Lambda)\mu^{\frac{2}{d+2}}\ \hbox{ and } \  |U| \geq c(d,\Lambda)\mu^{-\frac{d}{d+2}}. \]
\end{proof}

\subsection{Proof of relation \eref{Jgidentity}}\label{s.distJgcompute}
The equality \eref{Jgidentity} follows from the line of computation
\begin{align*}
J_g(\Omega,U_*,a) &= J_{\mu_*}(\Omega,a)+\int_{\Omega} \mu_* \omega\left(\frac{\rho(x, U_*)}{|U_*|^{\frac{1}{d}}}\right) dx\\
& = J_{\mu_*}(\Omega,a)+\mu_*\left[\int_{\Omega \setminus U_*}  \omega\left(\frac{d(x, \partial U_*)}{|U_*|^{\frac{1}{d}}}\right)  dx - \int_{\Omega \cap U_*} \omega\left(\frac{d(x, \partial U_*)}{|U_*|^{\frac{1}{d}}}\right) dx\right]\\
&= J_{\mu_*}(\Omega,a)+\mu_*\left[\int_{(\Omega \setminus U_*) \cup (U_* \setminus \Omega)}  \omega\left(\frac{d(x, \partial U_*)}{|U_*|^{\frac{1}{d}}}\right) dx   - \int_{(\Omega \cap U_*) \cup (U_* \setminus \Omega)} \omega\left(\frac{d(x, \partial U_*)}{|U_*|^{\frac{1}{d}}}\right) dx\right]\\
&=J_{\mu_*}(\Omega,a) + \mu_*\textup{dist}_\omega(\Omega,U_*) - \mu_*\int_{U_*} \omega\left(\frac{d(x, \partial U_*)}{|U_*|^{\frac{1}{d}}}\right) \ dx.
\end{align*}

\section{Proofs of standard technical lemmas}

\subsection{Relationship of the density quantities in \dref{scale-inv-densities}}\label{s.kappa-relation}

\begin{proof}
1.  (Bound $\sup_{z \in \Omega, r >0} \frac{|B_r(z) \cap \Omega|}{|B_{r/2}(z) \cap \Omega|}$) Let $z \in \Omega$, if $B_{r/4} \subset \Omega$ then we can upper bound by $|B_r|/|B_{r/4}| \leq C(d)$.  Otherwise there is $x_0 \in \partial \Omega \cap B_{r/4}$.

First suppose $r \leq 4 |\Omega|^{\frac{1}{d}}$.  Then
\[ |B_{r/2}(z) \cap \Omega| \geq |B_{r/4}(x_0) \cap \Omega| \geq c\kappa_0 |B_{r}|.\]

If $r/4  \geq r_* = |\Omega|^{\frac{1}{d}}$ then apply the density estimates to find
\[ |B_{r/2}(z) \cap \Omega| \geq |B_{r_*}(x_0) \cap \Omega| \geq \kappa_0|B_{r_*}| \geq c\kappa_0 |\Omega|.\]
Since the numerator is bounded above by $|\Omega|$ we are done.

2. (Bound $\sup_{z \in \partial \Omega, r>0} \frac{|B_r|}{|B_r(z) \setminus \Omega|}$) If $r \leq C|\Omega|^{\frac{1}{d}}$
\[ |B_r(z) \setminus \Omega| \geq |B_{C^{-1}r}(z) \setminus \Omega| \geq C^{-d}\kappa_0 |B_r|\]
while if $r \geq C|\Omega|^{\frac{1}{d}}$ then (for sufficiently large dimensional $C$)
\[ |B_r(z) \setminus \Omega| \geq \frac{1}{2} |B_r|.\]
\end{proof}
\subsection{Measure to Hausdorff bound}\label{s.Linftyupgrade}

The next result says that under inner/outer density estimates, \dref{scale-inv-densities}, we can upgrade estimates in measure to estimates in Hausdorff distance (for the domain). 
\begin{lemma}\label{l.Linftyupgrade}
Suppose that $\Omega$ and $\Omega'$ are sets with $|\Omega| = |\Omega'|$ both satisfy the inner/outer density estimates \dref{scale-inv-densities} then
 \[  d_H(\partial \Omega,\partial \Omega') \leq C\kappa_0^{-\frac{1}{d}}|\Omega \Delta \Omega'|^{1/d}\]
 as long as $|\Omega \Delta \Omega'| < \kappa_0|B_1|\min\{|\Omega|,|\Omega'|\}$.
\end{lemma}

This is standard but for completeness we include the proof.
\begin{proof}  
 Let $x\in \partial \Omega \setminus \overline{\Omega'}$ and call $r = d(x,\overline{\Omega'})$.  
 
 If $r \leq |\Omega|^{1/d}$ then
\[ |\Omega \setminus \Omega'| \geq |\Omega \cap B_r(x)| \geq \kappa_{0} |B_r|\]
so $r \leq C\kappa_0^{-1/d}|\Omega \Delta \Omega'|^{1/d}$.

If $r \geq |\Omega|^{1/d}$ then 
\[ |\Omega \setminus \Omega'| \geq |\Omega \cap B_r(x)| \geq \kappa_0|B_1||\Omega| \]
which contradicts the smallness assumption of the statement.

Next let $x \in \partial \Omega \cap \Omega'$ and let $r = d(x,\partial \Omega')$.  By the outer density estimate of $\Omega$
\[  |\Omega' \setminus \Omega| \geq |\Omega^C \cap B_r(x)| \geq \kappa_0 |B_r|\]
 since this side actually holds for arbitrary $r$ we are done.

\end{proof}

\subsection{Proof of \lref{dist-omega-est}}\label{s.dist-omega-est} 
Call $d(x) = d(x,\partial B)$ which has $|\grad d(x)| = 1$ a.e. and we split, for some $0 < t \leq t_0$ to be picked later,
\[ |A \Delta B| = |(A \Delta B )\cap\{d(x) \leq t\}| + |(A \Delta B )\cap\{d(x) \geq t\}|.\]
By the regularity assumption on $\partial B$
\[|(A \Delta B )\cap\{d(x) \leq t\}| \leq |\{d(x) \leq t\}| \leq P|B|^{\frac{d-1}{d}}t. \]
For the other term, by co-area formula twice,
\begin{align*}
 |(A \Delta B )\cap\{d(x) \geq t\}| &= \int_{t}^\infty \mathcal{H}^{d-1}((A \Delta B) \cap \{d(x) = s\}) \ ds\\
 &\leq t^{-\beta}\int_{t}^\infty s^{\beta}\mathcal{H}^{d-1}((A \Delta B) \cap \{d(x) = s\}) \ ds \\
 &= |B|^{\beta/d} t^{-\beta}\int_{\{ d(x) \geq t\}} |B|^{-\beta/d}d(x)^\beta{\bf 1}_{A \Delta B} \ dx \\
 & \leq |B|^{\beta/d}t^{-\beta}\textup{dist}_{\beta}(A,B).
 \end{align*}
 So combining these estimates
 \[ \frac{|A \Delta B|}{|B|}  \leq P|B|^{-\frac{1}{d}}t + \frac{\textup{dist}_{\beta}(A,B)}{t^\beta|B|^{1-\beta/d}}\]
 and we pick $t = P^{-\frac{1}{1+\beta}}|B|^{\frac{1}{d}-\frac{1}{1+\beta}}\textup{dist}_{\beta}(A,B)^{\frac{1}{1+\beta}}$ to give
 \[ \frac{|A \Delta B|}{|B|} \leq 2P^{\frac{\beta}{1+\beta}}\left(\frac{\textup{dist}_\beta(A,B)}{|B|}\right)^{\frac{1}{1+\beta}}\]
 
 For $\omega(r) = |\log (2+ r^{-1})|^{-p}$ do the same argument resulting in
 \[ \frac{|A \Delta B|}{|B|}  \leq P|B|^{-\frac{1}{d}}t + |\log (2+\frac{|B|^{\frac{1}{d}}}{t})|^{p}\frac{\textup{dist}_{\omega}(A,B)}{|B|}\]
 and choose $t = P^{-1}|B|^{\frac{1}{d}-1}\textup{dist}_{\omega}(A,B)$ to get the claim.
 
 Finally for $\omega(r) = |\log (2+ r^{-1})|^{-p} \wedge \gamma$ again do the same argument resulting in
 \[ \frac{|A \Delta B|}{|B|}  \leq P|B|^{-\frac{1}{d}}t + \left[|\log (2+\frac{|B|^{\frac{1}{d}}}{t})|^{p}\vee \gamma^{-1}\right]\frac{\textup{dist}_{\omega}(A,B)}{|B|}.\]
 Choose, as before, $t = P^{-1}|B|^{\frac{1}{d}-1}\textup{dist}_{\omega}(A,B)$ to get
 \[\frac{|A \Delta B|}{|B|}  \leq C\frac{\textup{dist}_{\omega}(A,B)}{|B|}(\gamma^{-1}+|\log(2+\frac{P|B|}{\textup{dist}_{\omega}(A,B)})|^{p}).\]

\qed

  \bibliographystyle{plain}
\bibliography{eigenvalue-articles}
\end{document}